\documentclass[a4paper,oneside,11pt]{amsart}

\usepackage[english]{babel}
\usepackage{amsmath,amstext,amsopn,amssymb, amsfonts, amsthm, mathtools}
\usepackage[dvipsnames]{xcolor}
\usepackage{geometry}
\usepackage{bbm}
\usepackage[colorlinks, citecolor = blue]{hyperref}
\usepackage{tcolorbox}
\usepackage[shortlabels]{enumitem}
\usepackage{esint}
\usepackage{scalerel}  
\usepackage{tikz}
\usepackage{float}
\usepackage{graphicx}  
\usepackage{mathrsfs}

\usepackage{todonotes}

\newtheorem{theorem}{Theorem}[section]
\newtheorem{lemma}[theorem]{Lemma}
\newtheorem{proposition}[theorem]{Proposition}
\newtheorem{corollary}[theorem]{Corollary}

\newtheorem{ltheorem}{Theorem}

\theoremstyle{definition}
\newtheorem{definition}[theorem]{Definition}

\theoremstyle{remark}
\newtheorem{remark}[theorem]{Remark}

\numberwithin{equation}{section}

\newtheorem*{definition*}{Definition}
\newtheorem*{theorem*}{Theorem}
\newtheorem{proposition*}{Proposition}
\newtheorem*{lemma*}{Lemma}
\newtheorem*{corollary*}{Corollary}
\newtheorem*{remark*}{Remark}

\DeclareMathOperator*{\esssup}{ess\,sup}

\DeclareMathOperator{\loc}{loc}

\newcommand{\Hardy}{H}
\DeclareMathOperator{\dist}{dist}

\DeclareMathOperator{\ind}{\mathbf{1}}
\DeclareMathOperator{\supp}{supp}

\newcommand{\N}{\ensuremath{\mathbb{N}}}
\newcommand{\Z}{\ensuremath{\mathbb{Z}}}

\newcommand{\R}{\ensuremath{\mathbb{R}}}
\newcommand{\C}{\ensuremath{\mathbb{C}}}

\newcommand{\mc}{\mathcal}
\newcommand{\ms}{\mathscr}

\DeclarePairedDelimiter\abs{\lvert}{\rvert}

\DeclarePairedDelimiter\cbrace\{\}
\DeclarePairedDelimiter\br()
\DeclarePairedDelimiter{\ip}\langle\rangle
\DeclarePairedDelimiter{\nrm}\lVert\rVert

\DeclarePairedDelimiter{\ceil}{\lceil}{\rceil}

\newcommand{\nrmb}[1]{\bigl\|#1\bigr\|}
\newcommand{\absb}[1]{\bigl|#1\bigr|}
\newcommand{\brb}[1]{\bigl(#1\bigr)}
\newcommand{\cbraceb}[1]{\bigl\{#1\bigr\}}
\newcommand{\ipb}[1]{\bigl\langle#1\bigr\rangle}

\newcommand{\nrmB}[1]{\Bigl\|#1\Bigr\|}
\newcommand{\absB}[1]{\Bigl|#1\Bigr|}
\newcommand{\brB}[1]{\Bigl(#1\Bigr)}

\newcommand{\bracB}[1]{\Bigl[#1\Bigr]}

\newcommand{\dd}{\hspace{2pt}\mathrm{d}}

\newcommand{\D}{\mathscr{D}}
\newcommand{\F}{\mathcal{F}}

\newcommand{\cexp}{\mathsf{E}}
\newcommand{\diff}{\mathsf{D}}
\newcommand{\per}{\mathsf{P}}
\newcommand{\MaxPer}{\mathcal{P}}
\newcommand{\MaxAvg}{\mathcal{M}}

\newcommand{\Ss}{\mathcal{S}}
\newcommand{\SpOp}{\mathcal{A}}
\newcommand{\CancSpOp}{\mathcal{C}}
\newcommand{\Car}{\mathrm{car}}

\renewcommand{\emptyset}{\varnothing}

\DeclareFontFamily{U}{wncy}{}
\DeclareFontShape{U}{wncy}{m}{n}{<->wncyr10}{}
\DeclareSymbolFont{mcy}{U}{wncy}{m}{n}
\DeclareMathSymbol{\Sh}{\mathord}{mcy}{"58} 

\allowdisplaybreaks

\begin{document}

\title{Cancellative sparse domination}

\author[Conde-Alonso]{Jos\'{e} M. Conde Alonso}
\address{Jos\'{e} M. Conde Alonso \hfill\break\indent 
 Departamento de Matem\'aticas \hfill\break\indent 
 Universidad Aut\'onoma de Madrid \hfill\break\indent 
 C/ Francisco Tom\'as y Valiente sn\hfill\break\indent 
 28049 Madrid, Spain}
\email{jose.conde@uam.es}

\author[Lorist]{Emiel Lorist}
\address{Emiel Lorist \hfill\break\indent 
Delft Institute of Applied Mathematics \hfill\break\indent
Delft University of Technology \hfill\break\indent
P.O. Box 5031 \hfill\break\indent
2600 GA Delft, The Netherlands}
\email{e.lorist@tudelft.nl}

\author[Rey]{Guillermo Rey}
\address{Guillermo Rey \hfill\break\indent 
 Departamento de Matem\'aticas \hfill\break\indent 
 Universidad Aut\'onoma de Madrid \hfill\break\indent 
 C/ Francisco Tom\'as y Valiente sn\hfill\break\indent 
 28049 Madrid, Spain}
\email{guillermo.rey@uam.es}

\thanks{Conde-Alonso was partially supported by grants \texttt{CNS2022-135431} (Ministerio de Ciencia, Spain) and by the 2025 Leonardo Grant for Scientific Research and Cultural Creation \texttt{LEO25-1-19094-CBB-MAT-169} (BBVA Foundation). The BBVA Foundation accepts no responsibility for the opinions, statements and contents included in the project and/or the results thereof, which are entirely the responsibility of the authors. Lorist was partially financed by the Dutch Research Council (NWO) on the project ``The sparse revolution for stochastic partial differential equations'' with project number \href{https://doi.org/10.61686/ZGRMR99948}{VI.Veni.242.057}. Rey was partially supported by grants \texttt{PID2022-139521NA-I00} and \texttt{RYC2024-051323-I}, both funded by \texttt{MICIU/AEI/10.13039/501100011033}.
}
	
\begin{abstract}
We present a general sparse domination principle which respects the cancellative structure of the functions under study. We obtain sparse domination results in filtered measure spaces, including martingale settings in one and two parameters, as well as in the Euclidean setting. In the one parameter martingale setting, we obtain a sparse characterization of the $H^1$-norm. The proofs make critical use of precise level-set estimates for generalized medians. Our results imply new, quantitatively sharp, weighted results for martingale transforms and Calderón-Zygmund operators from $H^p(w)$ to $L^p(w)$.
\end{abstract}
	

\maketitle

\section*{Introduction}

Many results in analysis involve the precise description of certain features of functions: size, speed of growth, mass, oscillation... A modern technique to study the size of functions, which has been used very frequently in harmonic analysis, is sparse domination. The technique was initiated by Lerner \cite{Le13a,Le13b} as an alternative route to the $A_2$ theorem, based on domination of functions instead of their representation in terms of appropriate dyadic objects \cite{Pe07,Hy12}. The subsequent developments over the past fifteen years have resulted in precise weighted bounds in $L^p$ for numerous operators, including martingale transforms \cite{Lac2017} and Calder\'on-Zygmund operators \cite{CR16,LN15}, and even some new characterizations of endpoint estimates \cite{CCDPO17}, to cite a few applications. 

Sparse domination for an operator $T$ is an estimate of the form
\begin{equation}\label{eq:ClassicalSparse} \tag{SizeSparse} 
|Tf| \lesssim \SpOp_\Ss |f|,
\end{equation}
where the inequality may hold pointwise or in some weaker form. The operator $\SpOp_\Ss$ is an averaging operator over a family of sets $\Ss$:
$$
\SpOp_\Ss f(x) = \sum_{Q \in \Ss} \left|\langle f \rangle_Q\right| \ind_Q(x),
$$
where $\ip{f}_Q:=\frac{1}{\abs{Q}}\int_Qf$.
The family $\Ss$ depends on $f$ but is $\eta$-sparse for some $0<\eta<1$  independent of $f$.
One of the equivalent ways to define an $\eta$-sparse collection $\mathcal{S}$ is to say that for each $Q\in\Ss$ there exists $E_Q \subseteq Q$ such that $|E_Q| \geq \eta |Q|$ and the sets $E_Q$ are pairwise disjoint. 

Sparse domination has been a very successful tool to analyze the $L^p$-norms of operators and functions, and especially so to tackle delicate quantitative questions pertaining to the  (weak) $L^1$-norm. All available sparse domination techniques are critically based on an estimate that does not see any cancellation present in the function $f$. Indeed, sparse families are generally constructed using the level sets of an operator for which a local weak-type estimate is available, which is the key to averaging over a family that ends up being sparse. The process eliminates any information about the cancellation --or the smoothness-- of the functions involved. While this is harmless for strong- or weak-type estimates, one cannot hope to get any results applicable to spaces in which cancellation is essential, like the Hardy space $\Hardy^1$. Indeed, $\SpOp_\Ss|f|$ is in general not an $L^1$ function for $f \in \Hardy^1$, so \eqref{eq:ClassicalSparse} cannot recover the $\Hardy^1 \to L^1$ boundedness of $T$. The problem is easiest to illustrate in the dyadic setting: if we let $\ms{D}$ be the standard dyadic grid and define
\begin{equation}\label{eq:dyadMax}\tag{MaxD}
\MaxAvg_\D f(x) = \sup_{x\in Q\in\D} |\langle f \rangle_Q|,    
\end{equation}
then $\|f\|_{\Hardy^1_\D} = \|\MaxAvg_\D f\|_1$, and one is immediately tempted to search for sparse domination results of the form
$$
\MaxAvg_\D f \lesssim \SpOp_\Ss f.
$$
We shall see in the body of the paper that such an estimate cannot hold in general. Nevertheless, our goal is to give a form of sparse domination that still preserves the cancellation properties of the function. What we shall do is to reinforce the right-hand side of the desired inequality by precomposing with an operator which is bounded on $L^p$ for $0<p<\infty$. Said operator is the percentile maximal function. Given $0<r<1$, a set of positive and finite measure $\Omega$ and a locally integrable function $f$, we set
$$
\per_\Omega^r(f) := \inf\Big\{\lambda \in \R:\, |\{x \in \Omega:\, f(x) > \lambda \}| \leq r|\Omega|\Big\}.
$$
We have 
\begin{align*}
      |\{x \in \Omega:\, f(x) > \per_\Omega^r(f)\}| &\leq r|\Omega|, \\
   |\{x \in \Omega:\, f(x) < \per_\Omega^r(f)\}| &\leq (1-r)|\Omega|, 
\end{align*}
so $\per_\Omega^{\frac12}(f)$ is nothing but a choice of a median of $f$ over $\Omega$. Given any value of $r$, $f$ lies below $\per^r_\Omega(f)$ on a portion of $\Omega$ of measure $(1-r)|\Omega|$. Therefore, we interpret $\per^{1-r}_\Omega(f)$ as the $r$-th percentile of $f$ over $\Omega$. We form the maximal function
$$
\MaxPer^r_\D f(x) = \sup_{x\in Q\in \D} \per_Q^r(|f|).
$$
When no confusion may arise, the precise value of $r$ is both fixed and unimportant, so we omit it from the notation. 

Our first sparse domination result can be applied to many natural operators in dyadic or, more generally, martingale settings. For the latter context, we postpone introducing the terminology and associated technicalities to the body of the paper. Likewise, we postpone the definition of the operators under consideration until we specifically deal with them.

\begin{ltheorem}\label{th:theoremA}
    Let $T$ be a martingale transform or a Haar shift operator. There exists $r \in (0,1)$ such that for each bounded $f$, there is a sparse $\Ss \subseteq \D$ so that
    $$
    \abs{Tf(x)} \lesssim \sum_{Q\in \Ss} \per_Q^r(\MaxAvg_Q f) \ind_Q(x), \qquad x \in \R^d.
    $$
\end{ltheorem}

Above, $\MaxAvg_Q$ denotes the restriction of $\MaxAvg_\D$ where the supremum runs over over dyadic subcubes of $Q$. The easy chain of inequalities
$$
|\langle f \rangle_Q| \leq \per_Q^r(\MaxAvg_Q(f)) \leq \tfrac{1}{r} \langle |f|\rangle_Q
$$
shows that Theorem \ref{th:theoremA} is an improvement over the usual sparse domination statements, while at the same time it avoids formulations that are too strong to hold in general. Its proof is based on the local inequality
$$
\per_Q^{Cr}(\abs{Tf}) \lesssim_r \per_Q^r(\MaxAvg_Q(f)), \qquad \mbox{for some }C>1.
$$
We use it to bypass the weak-type $(1,1)$ inequality for the non-cancellative version of $\MaxAvg_\D$, which is used in one way or another in all of the existing sparse domination principles and which makes cancellative estimates hopeless. We shall prove Theorem \ref{th:theoremA} via the corresponding estimates for maximal truncations of the operators under study. We shall show in the body of the paper that the composite sparse operator 
$$
\sum_{Q\in \Ss} \per_Q^r(\MaxAvg_Q f) \ind_Q(x)
$$
maps $\Hardy_\D^1(\R^d)$ to $L^1(\R^d)$. This is because $\MaxPer^r_\D$ is bounded on $L^p(\R^d)$ for all $0<p\leq \infty$, and therefore 
$$
\CancSpOp_\Ss f(x) := \sum_{Q\in \Ss} \per_Q^r(|f|) \ind_Q(x) 
$$
is bounded on $L^p(\R^d)$, $0<p<\infty$. Hence, Theorem \ref{th:theoremA} allows one to recover the known Hardy space endpoint boundedness results for $T$. 

Theorem \ref{th:theoremA} is a cancellative $\ell^1$-estimate. In sparse domination, the easiest operators to study are often single-scale ones, of which maximal functions like the one above are the archetypal example. For the maximal operator in the non-cancellative setting one can easily prove the stronger $\ell^\infty$-estimate
$$
\sup_{x\in Q \in \D} \langle |f| \rangle_Q \lesssim \sup_{x\in Q \in \Ss} \langle |f| \rangle_Q,
$$
with $\Ss$ being sparse and dependent on $f$. One can prove an estimate similar to Theorem \ref{th:theoremA} for $\MaxAvg_\D$ in an $\ell^\infty$ sense and tied to the martingale structure of the dyadic system. We will however give an alternative, cancellative sparse domination result for $\MaxAvg_\D$ in terms of the sparse maximal function $\MaxAvg_\Ss$ --defined similarly to \eqref{eq:dyadMax} but with the sup taken over a sparse subfamily of $\D$--, which has additional interesting consequences. The first one is that it yields a sparse characterization of the dyadic $\Hardy^1$-norm.

\begin{ltheorem}\label{th:theoremB}
    Let $f\in \Hardy_\D^1(\R^d)$. Then
    $$
    \|f\|_{\Hardy_\D^1} \eqsim \sup_{\Ss \; \mbox{sparse}} \left\|\MaxAvg_\Ss f\right\|_{L^1}
    \eqsim \sup_{\Ss \; \mbox{sparse}} \left\|\SpOp_\Ss f\right\|_{L^1}.
    $$
\end{ltheorem}

As we said above, the proof of Theorem \ref{th:theoremB} uses a cancellative sparse $\ell^\infty$-domination result for $\MaxAvg_{\D}$. We emphasize that the corresponding non-cancellative estimates do not yield a useful characterization in this setting, and we shall provide evidence of this in the body of the paper. We have stated Theorems \ref{th:theoremA} and \ref{th:theoremB} for the (regular) dyadic martingale, but both hold in much wider generality. Theorem \ref{th:theoremA} holds for martingale transforms in any $\sigma$-finite filtered measure space with the additional property of regularity, the martingale analogue of the doubling property. We will show that the regularity hypothesis is necessary. Theorem \ref{th:theoremB} holds in the same setting, even without the regularity assumption. Finally, we also state and prove a natural analogue of Theorem \ref{th:theoremA} for the martingale square function, completing our analysis of martingale-related objects. There are a few technicalities involved in our general formulation, which we postpone to Section \ref{sec:sec1}. 

Our sparse domination principle for single-scale operators is very flexible: neither the regularity of the filtration is needed, nor is it important that we have a family of nested partitions. This allows us to deduce the second consequence of our alternative approach to cancellative sparse domination of maximal functions. Leveraging the fact that we do not need our operators to satisfy a weak-$L^1$ estimate, we can consider the (cancellative) strong maximal operator. This can be defined as follows: 
$$
\MaxAvg_{\D \times \D}f(x) = \sup_{\substack{
Q_1 \in \D(\R^{d_1}), Q_2 \in \D(\R^{d_2}): \\ x \in Q_1 \times Q_2 }} \left|\langle f \rangle_{Q_1\times Q_2}\right|.
$$
The corresponding maximal median operator $\MaxPer_{\D \times \D}^r$ is defined in the expected way, and we call any set $Q_1\times Q_2$ as above a dyadic rectangle. Our result reads as follows:

\begin{ltheorem}\label{th:MaxStrong}
For each $f\in L^\infty(\R^{d_1+d_2})$, there exists a sparse family $\Ss$ of dyadic rectangles such that
$$
\MaxAvg_{\D \times \D}f(x) \leq 2\,\MaxPer_{\D \times \D}^{\frac12}(\MaxAvg_\Ss f)(x), \qquad x \in \R^{d_1+d_2}. 
$$
\end{ltheorem}
The statement will be deduced from a general procedure, reminiscent of \cite{CorFeff75}, which allows one to always extract a sparse subfamily $\Ss$ from any countable family of sets $\mathcal{E}$. Theorem \ref{th:MaxStrong} is the first biparametric sparse domination result, and is in sharp contrast with the negative result from \cite{BCOR19}, which proved that the
non-cancellative sparse bound fails for the strong maximal function.

\medskip

In the second part of the paper, we turn our attention to cancellative estimates in the Euclidean setting. We restrict our attention to Calder\'on-Zygmund theory, although we expect that our ideas can be pushed to many other operators of interest which exhibit similar behavior near the $\Hardy^1 \to L^1$ endpoint, or to the setting of \cite{BFP16}. A Calder\'on-Zygmund operator is an $L^2$-bounded operator $T$ with an associated kernel $K$. We assume that $K$ satisfies a smoothness estimate of order $s>0$. When $0<s<1$, it reduces to the classical H\"older regularity of the kernel, while higher values of $s$ involve higher-order differentiability properties for $K$. The details are postponed to Section \ref{sec:sec3}. This assumption is by no means the weakest possible pointwise smoothness condition one can impose on $K$ for $L^p$-boundedness results in the Banach range, but it is important for obtaining boundedness results in $\Hardy^p(\R^d)$ for $0<p<1$. Calderón-Zygmund operators $T$ with kernels satisfying the smoothness assumptions that we use in this paper are known to be representable as averages of Haar shifts. However, we will justify in the body of the paper that said representation cannot be combined with Theorem \ref{th:theoremA} to get a sparse domination result for $T$. Instead, the role of $\MaxAvg_\D$ in the continuous theory is played by a different, cancellative grand maximal function. Given a cube $Q$, denote by $\mc{F}_s(Q)$ the family of $L^\infty$-normalized $s$-smooth bump functions supported on $Q$. Then we define
$$
\MaxAvg_{Q}^s f(x) = \sup_{x\in R\subseteq Q} \sup_{\varphi \in \mc{F}_s(R)} \frac{1}{|R|} \absB{\int_R f(y)\varphi(y) \dd y },
$$
which recovers the local Hardy--Littlewood maximal operator when $s\to 0$, as the supremum is in this case attained for $\varphi = \mathrm{sgn}(f) \cdot \ind_R$. The role of the above maximal function in the continuous setting is the same as that of the localized dyadic maximal function in the dyadic context. Our sparse domination result for Calder\'on-Zygmund operators reads as follows:

\begin{ltheorem}\label{th:theoremC}
Let $s >0$ and $T$ be an $s$-smooth Calder\'on-Zygmund operator. Let $f \in L^\infty_c(\R^d)$ be supported on a cube $Q_0 \subseteq \R^d$. Then there exists $r\in (0,1)$, only depending on $d$, and a $\tfrac12$-sparse family $\Ss\subseteq\D(Q_0)$ such that 
$$
|Tf(x)|\lesssim \sum_{Q\in \Ss}\per_Q^{r}(\MaxAvg_{3Q}^s(f)) \ind_Q(x), \qquad x \in Q_0.
$$
\end{ltheorem}

Theorem \ref{th:theoremC}'s statement is almost the same as that of Theorem \ref{th:theoremA}. As in that case, it recovers the $\Hardy^1\to L^1$ endpoint estimate for Calder\'on-Zygmund operators with H\"older smooth kernel. Moreover, one can also recover the boundedness from $\Hardy^p(\R^d)$ to $L^p(\R^d)$ when $p\in (d/(d+s),1]$ from the statement, which coincides exactly with the optimal range of the classical theory, using the $L^p$-boundedness of $\CancSpOp_\Ss$. This complements the known size results from sparse domination theory and can also be compared to the recent wavelet representation results that preserve smoothness, like \cite{HytonenLappas,DiPlinioWickWilliams2023}. The latter are useful to study, for example, Sobolev norms, while our result is better suited for cancellative estimates of Hardy-space type. 

The key technical novelty in our proof is the replacement of the standard determination of the sparse family as the level set of a maximal function --which must be non-cancellative for the method to work-- by a good-$\lambda$ type argument. After that step is performed, a careful application of classical Calder\'on-Zygmund methods allows us to extract the pointwise estimate in the statement. Theorem \ref{th:theoremC} can be proved by combining these ideas with the general sparse domination principle of \cite{LLO21}. 

\medskip

Sparse domination was originally developed as a tool to prove sharp weighted norm inequalities, most notably in connection with the so-called $A_2$-theorem \cite{Hy12,Le13a}. Our cancellative results, Theorems \ref{th:theoremA} and \ref{th:theoremC}, recover the familiar weighted estimates on $L^p(w)$ for $p\in (1,\infty)$ and $w \in A_p$, the class of Muckenhoupt weights. More interestingly, preserving the cancellation of the input functions allows one to consider estimates beyond the reach of previous sparse domination results, namely in the cases $p\leq 1$ and/or $w\notin A_p$. In these regimes, the operators we study are not bounded on $L^p(w)$, but Hardy space estimates are known qualitatively. In the body of the paper, we obtain sharp quantitative $L^p(w)\to L^p(w)$ results for variants of $\CancSpOp_\Ss$ and deduce new sharp weighted inequalities for our operators.

The rest of the paper is organized as follows: Section \ref{sec:sec1} contains our estimates for martingale maximal functions and the proof of Theorems \ref{th:theoremB} and \ref{th:MaxStrong}. The estimates for multiscale dyadic operators, like martingale transforms and Haar shifts, are in Section \ref{sec:sec2}. Section \ref{sec:sec3} is devoted to Calder\'on-Zygmund theory and other Euclidean results. Finally, in Section \ref{sec:sec4} we prove norm estimates for 
the cancellative sparse operators and deduce Hardy space estimates for the operators studied in Sections \ref{sec:sec2} and \ref{sec:sec3}.

\section{General martingales - single-scale operators} \label{sec:sec1}
We start by presenting our first approach to cancellative domination, which works for general martingales --and of course includes the usual dyadic setting in $\R^d$, as we will detail below--. We work in the generality of filtered probability spaces, which includes the dyadic Euclidean setting of the Introduction. Let $(\Omega,\F,\mu)$ be a probability space, equipped with a filtration of $\sigma$-algebras $\{\F_k\}_{k \geq 0}$ such that $\cup_{k\geq 0}\F_k$ generates $\F$. We denote the associated conditional expectations by
$$
f_k := \cexp_k[f] := \cexp[f|\F_k], \qquad f \in L^1(\Omega).
$$
As is standard, we identify an integrable function $f \in L^1(\Omega)$ with the martingale $\{f_k\}_{k\geq 0}$ that converges to it. We also use standard notation for martingale differences:
\begin{align*}
    df_k := \diff_k(f) := \begin{cases}
     \cexp_k[f] - \cexp_{k-1}[f] &\text{if } k \geq 1, \\
     \cexp_0[f] &\text{if } k = 0.
    \end{cases}
\end{align*}
Fix $0<r<1$ and a measurable function $f$. In the martingale setting, the natural analogue of the set-based percentile from the Introduction is a conditional expectation based version. Following \cite{Tomkins75, Tom78}, we define a $k$-th conditional percentile of $f$ at ratio $r$ as any $\F_k$-measurable function $\per_k^rf$ with the following property: for each $A \in \F_k$, the following two inequalities hold:
\begin{align*}
    \mu(\{x \in A:\, f(x) > \per_k^rf(x)\}) &\leq r\mu(A),\\
     \mu(\{x \in A:\, f(x) < \per_k^rf(x)\}) &\leq (1-r)\mu(A).
     \end{align*}
In general, there may be many possible choices for $\per_k^r$. For concreteness, we will choose the following version:
\begin{align*}
  \per_k^r f(x) := \inf\{ t \in \mathbb{Q}:\, \cexp_k[\ind_{f > t}](x) \leq r \}.
\end{align*}
In \cite{Tomkins75, Tom78}, only the case $r=\frac{1}{2}$ is considered, and $\per_k^{\frac{1}{2}} f$ is called the $k$-th conditional median of $f$. By \cite{Tom78} we have the pointwise differentiation theorem
$$
\per_k^r f \to f, \qquad k\to \infty,
$$
for all $0<r<1$. Also, by measurability, we have $\per_k^r \cexp_k [f] = \cexp_k [f]$ if $f \in L^1(\Omega)$. The key to our approach is the maximal function associated to the family $\{\per_k^r\}_{k\geq 0}$:
$$
\MaxPer^r f(x) := \sup_{k \geq 0} \per_k^r |f|(x), \qquad x \in \Omega.
$$
The dyadic version of $\MaxPer^r$ was studied in \cite{MO2011} when $r=\frac12$. In most instances, the parameter $r$ will be fixed from the start and we shall omit it from the notation, writing $\per_k f$ and $\MaxPer f$ instead of $\per_k^r f$ and $\MaxPer^r f$, respectively. We will also be using truncated versions of $\MaxPer$, namely for $\ell\geq 0$
$$
\MaxPer_\ell f(x) := \sup_{k \leq \ell} \per_k^r |f|(x), \qquad x \in \Omega.
$$

\begin{lemma}\label{lem:levelSetMaxPer}
For measurable $f$ we have
  \begin{align*}
    \mu\brb{\{\MaxPer f > \lambda\}} \leq \tfrac{1}{r} \ \mu\brb{\{|f| > \lambda\}}, \qquad \lambda \geq 0.
  \end{align*} 
In particular, for $p \in (0,\infty)$ we have 
$$\|\MaxPer\|_{L^p(\Omega)\to L^p(\Omega)} \leq r^{-\frac{1}{p}}.$$
\end{lemma}

\begin{proof}
  Without loss of generality, assume that $f \geq 0$. Define
    $$
    \Omega_\ell :=\{ \MaxPer_\ell f > \lambda\}\cap \{\MaxPer_{\ell-1} f \leq \lambda \}, \qquad \ell \geq 0,
    $$
    using the convention $\MaxPer_{-1} f=0$.
  Since $\Omega_\ell \in \F_\ell$,  we have
$$
\mu\brb{\{x \in \Omega_\ell:\, f(x) \leq \lambda\}} \leq \mu\brb{\{x \in \Omega_\ell:\, f(x) < \per_\ell f(x)\}}\leq  (1-r)\mu(\Omega_\ell),
$$
and thus 
$$
\mu\brb{\{x \in \Omega_\ell:\, f(x) > \lambda\}} \geq   r\mu(\Omega_\ell).
$$
Therefore, we can compute   
\begin{align*}
    \mu\brb{\{\MaxPer f > \lambda\}} &= \sum_{\ell \geq 0} \mu(\Omega_\ell) 
    \leq \sum_{\ell \geq 0} \frac{1}{r}\ \mu\brb{\{x \in \Omega_\ell: f(x) > \lambda\}} \\
    &\leq \frac{1}{r}\ \mu\brb{\{f > \lambda\}},
\end{align*}
as we wanted. The second claim follows via the layer cake formula.
\end{proof}

We now turn to the main goal of this section. The percentile maximal function $\MaxPer$ is the enlargement that enables cancellative sparse domination of Doob's maximal function, which for integrable $f$ is defined as
$$
\MaxAvg f := \sup_{k\geq 0} \left|\cexp_k[f]\right|.
$$
As with $\MaxPer$, we also define the truncated version by
$$
\MaxAvg^\ell f  = \sup_{0\leq k \leq \ell} \left|\cexp_k[f]\right|.
$$
A sequence of sets $\{A_k\}_{k\geq 0}$ is said to be adapted to the filtration if $A_k \in \F_k$ for each $k$. For our definition of sparsity, we take inspiration from \cite{DPS23}. 
\begin{definition}\label{def:sparsemart}
    For $0 < \eta \leq 1$, an adapted sequence $\Ss=\{S_k\}_{k\geq 0}$ is said to be $\eta$-sparse if for every $k \geq 0$ and every $\F_k$-measurable set $A \subseteq S_k$ we have 
\begin{align*}
\mu\Big(A \setminus \bigcup_{m \geq k+1} S_m \Big) \geq \eta \mu(A).
\end{align*}
\end{definition} 
For brevity, we refer to $\frac{1}{2}$-sparse sequences as sparse. We can give an alternative definition of $\eta$-sparsity using conditional expectations. Indeed, define $S_{\geq k} := \bigcup_{m \geq k} S_m$. Then $\Ss$ is $\eta$-sparse if and only if for every $k \geq 0$
\begin{align*}
    \cexp_k[\ind_{S_{\geq k+1}^c}] &\geq \eta  &&\mu\text{-a.e. on }S_k,
    \intertext{or equivalently if }
    \cexp_k[\ind_{S_{\geq k+1}}] &\leq 1-\eta &&\mu\text{-a.e. on }S_k.
\end{align*}

 We next show that our definition is equivalent to that of Definition 2 in Section 3 of \cite{DPS23}. There, an increasing sequence $\{\nu_k\}_{k\geq 0}$ of stopping times is called sparse if the sets $E_k = \{\nu_k < \infty\}$ satisfy the following: for every  $A \subseteq E_k$ such that $A \in \F_{\nu_k}$, 
$$
\mu(A \cap E_{k+1}) \leq \tfrac{1}{2} \mu(A).
$$
This is equivalent to saying that for each $k$
\begin{align} \label{sparse_seq_of_sts}
    \cexp[\ind_{E_{k+1}} \, |\, \F_{\nu_k}] \leq \tfrac{1}{2} \quad \text{$\mu$-a.e. on $E_k$}.
\end{align}
From this, it is clear that if for a sparse sequence $\{S_k\}_{k\geq 0}$ we define
\begin{align*}
  \nu_0(x) &:= \inf \{m\geq 0:\, x \in S_m\}, &&x \in \Omega,\\
  \nu_{k}(x) &:= \inf \{ m > \nu_{k-1}(x):\, x \in S_m \} &&x \in \Omega, \,k \geq 1,
\end{align*}
then $\{\nu_k\}_{k\geq 0}$ is an increasing, sparse sequence of stopping times. The converse implication is contained in the following lemma.

\begin{lemma}\label{lem:defsOfSparse}
  Let $\{\nu_k\}_{k\geq 0}$ be a sparse sequence of stopping times. Define
  \begin{align*}
    S_m := \bigcup_{k \geq 0} \{\nu_k = m\}.
  \end{align*}
  Then, $\{S_m\}_{m\geq 0}$ is sparse.
\end{lemma}
\begin{proof}
  Fix $m\geq 0$ and partition $S_m$ into the disjoint subsets
  $$
B_k:= \cbrace{\nu_k=m}\cap \cbrace{\nu_{k-1}<m}, \qquad k\geq 0
  $$
  using the convention $\nu_{-1}=-1$. We note that
    \begin{align}\label{eq:sparseequiv}
    B_k \cap S_{\geq m+1} &\subseteq \{\nu_{k+1} < \infty\} =E_{k+1},
  \end{align}
since for $x \in B_k$ we have $\nu_j(x) \leq  m$ for all $j \leq  k$, so for $x \in B_k\cap S_{\geq m+1}$ there must be a $j > k$ such that $\nu_j(x) < \infty$ and hence $\nu_{k+1}(x)<\infty$. 

We need to show
  \begin{align*}
    \cexp_m[\ind_{S_{\geq m+1}}] \leq \tfrac{1}{2} \qquad \text{$\mu$-a.e. on $S_m$},
  \end{align*}
for which it suffices to show that $\cexp_m[\ind_{S_{\geq m+1}}] \leq \frac{1}{2}$ $\mu$-a.e. on $B_k$ for every $k\geq 0$. This follows from
  \begin{align*}
    \ind_{B_k} \cexp_m[\ind_{S_{\geq m+1}}] &\leq \ind_{B_k}\cexp_{m}[\ind_{E_{k+1}}] \\
    &= \ind_{B_k}\cexp[\ind_{E_{k+1}} \,|\, \mathcal{F}_{\nu_k}] \leq \tfrac{1}{2} \ind_{B_k},
  \end{align*}
  where we have used that $B_k \in \mc{F}_m$ and \eqref{eq:sparseequiv} in the first step, $B_k \subseteq \{\nu_k = m\}$ in the second step and sparsity of $\{\nu_k\}_{k\geq 0}$ in the third step.  
\end{proof}

The preceding discussion shows that the notion of sparsity for adapted sequences $\cbrace{S_k}_{k\geq 0}$ is equivalent to the sparsity of the associated sequence of stopping times $\cbrace{\nu_k}_{k\geq 0}$ in the sense of \cite{DPS23}. We will therefore move freely between these two viewpoints, depending on which is more convenient in a given argument. Given an $\eta$-sparse sequence $\Ss=\cbrace{S_k}_{k\geq 0}$, we define the corresponding sparse maximal operator and sparse operator by
\begin{align*}
    \MaxAvg_\Ss f(x) &:= \sup_{k\geq 0} \ind_{S_k}(x) |\cexp_k [f](x)|, &&x \in \Omega,\\
    \SpOp_\Ss f(x) &:= \sum_{k\geq 0} \ind_{S_k}(x) |\cexp_k [f](x)|, &&x \in \Omega.
\end{align*}
Sparse operators are defined in terms of $\ell^1$-sums, while sparse maximal operators are $\ell^\infty$-type quantities. It is immediate that 
$$
\MaxAvg_\Ss f(x) \leq \SpOp_\Ss f(x),\qquad x \in \Omega.
$$
The following result is the key to our approach.

\begin{proposition}\label{prop:keyIneqMaximals}
    Let $f \in L^\infty(\Omega)$. There exists a sparse sequence $\Ss$ such that 
    $$
    \MaxAvg f  \leq 2\, \MaxPer^{\frac12}\br{\MaxAvg_\Ss f }. 
    $$
\end{proposition}
\begin{proof}
  The proof strategy is as follows: we construct the sparse family by adding layers which dominate the different dyadic scales of the range of the maximal operator, starting from the largest and proceeding downward. At each scale, we retain only those sets whose conditional overlap with previous scales is small; points we discard are automatically controlled by $\MaxPer(\MaxAvg_\Ss f)$. For simplicity, we assume that $\|f\|_{L^\infty(\Omega)} = 1$. For each $k \geq 0$ and $\ell\geq 1$ define
  \begin{align*}
    \Omega_k^\ell = \{ x \in \Omega:\, \MaxAvg^k f(x) > 2^{-\ell} \text{ and } \MaxAvg^{k-1} f(x) \leq 2^{-\ell}\},
  \end{align*}
  where we use the convention $\MaxAvg^{-1}f=0$.
    For each $\ell \geq 1$, $\{\Omega_k^\ell\}_{k\geq 0}$ is a partition of $\{\MaxAvg f>2^{-\ell} \}$. Next, for
each $k\geq 0$, we initialize by setting $S_k^1 := \Omega_k^1$. We will define $S_k^{\ell+1}$ inductively for $\ell \geq 1$. Define 
  \begin{align*}
    S^\ell := \bigcup_{\substack{k \geq 0,\, m \leq \ell}} S_k^m
  \end{align*}
and set
  \begin{align*}
    S_k^{\ell+1} &:= \left\{ x \in \Omega_k^{\ell+1}:\, \cexp_k[\ind_{S^\ell}](x) \leq  \tfrac{1}{2} \right\}.
  \end{align*}
By construction, we have 
$$
S_{k}^{\ell+1} \cap \Big[\bigcup_{\substack{n \leq k, \,m \leq \ell}} S_n^m \Big]= \emptyset,
$$
because the second set on the left-hand side of the above equation is $\mathcal{F}_k$-measurable. This means that
$$
S_k^\ell \cap S_{k'}^{\ell'} = \emptyset \quad \mbox{whenever } k\leq k' \mbox{ and } \ell < \ell'.
$$ 
Finally, we put $S_k = \cup_{\ell \geq 1} S_k^\ell$ and we claim that $\Ss=\{S_k\}_{k\geq 0}$ is sparse. Indeed, the above equality yields
\begin{align*}
\ind_{S_k^\ell} \cexp_k[\ind_{\cup_{j>k} S_j}] & =\cexp_k[\ind_{S_k^\ell}\ind_{\cup_{j>k} S_j}] \\
& = \cexp_k\Big[\ind_{S_k^\ell}\ind_{\bigcup_{j>k,m < \ell} S_j^m}\Big] \\
& \leq \cexp_k\Big[\ind_{S_k^\ell}\ind_{S^{\ell-1}}\Big] \leq \tfrac{1}{2} \ind_{S_k^\ell},
\end{align*}
where we use  ${S}^{0} =\emptyset$.
Summing over $\ell$ yields the sparsity of $\Ss$. 

We now prove the domination. Fix $k\geq 0$ and $x\in\Omega$, and define $\ell$ by
$$
2^{-\ell} < |\cexp_k [f](x)| \leq 2^{-\ell+1}.
$$
Then, there exists $m\leq k$ such that $x \in \Omega_m^\ell$, and there are two possibilities. If $x \in S_m^\ell$, then $\MaxAvg_{\Ss}$ includes a dominating average. Indeed,
\begin{align*}
|\cexp_k [f](x)| \leq 2^{-\ell + 1} &\leq 2 \,|\cexp_m[f](x)| = 2\,|\ind_{S_m}(x) \cexp_m[f](x)| \leq 2 \,\MaxAvg_{\Ss}f(x).
\end{align*}
By the pointwise differentiation theorem above we have $\per_j^{\frac12} g \to g$ a.e. as $j\to\infty$ for $g=\MaxAvg_\Ss f$ and therefore $\MaxAvg_\Ss f(x) \leq \MaxPer^{\frac12}(\MaxAvg_\Ss f)(x)$. We conclude that $|\cexp_k [f](x)| \leq 2\,\MaxPer^{\frac12}(\MaxAvg_\Ss f)(x)$. Otherwise, we must have $x \in \Omega_m^\ell \setminus S_m^{\ell}$ and so
$$
\cexp_m[\ind_{S^{\ell-1}}](x) > \tfrac12.
$$
  On $S^{\ell-1}$, $\MaxAvg_\Ss f \geq 2^{-\ell+1}$. This gives $\per_m^{\frac12}(\MaxAvg_\Ss f)(x) \geq 2^{-\ell+1}$, and so
$$
|\cexp_k [f](x)| \leq 2^{-\ell+1} \leq \per_m^{\frac12}(\MaxAvg_\Ss f)(x) \leq \MaxPer^{\frac12}(\MaxAvg_\Ss f)(x).
$$
Taking the supremum over $k$ in both cases yields $\MaxAvg f(x) \leq 2\, \MaxPer^{\frac12}(\MaxAvg_\Ss f)(x)$.
\end{proof}

We are now ready to prove Theorem \ref{th:theoremB}, which characterizes the $\Hardy^1$-norm via sparse operators. The key point is that Proposition \ref{prop:keyIneqMaximals} makes the forward direction short.

\begin{proof}[Proof of Theorem \ref{th:theoremB}]
  We first show that
  \begin{align*}
    \|\MaxAvg f\|_{L^1(\Omega)} \lesssim \left\| \SpOp_\Ss f \right\|_{L^1(\Omega)}
  \end{align*}
  for some sparse sequence $\Ss$. By density it suffices to consider $f \in L^\infty(\Omega)$. Then by Proposition \ref{prop:keyIneqMaximals}, there exists a sparse $\Ss=\{S_k\}_{k\geq 0}$ such that $\MaxAvg f \leq 2\, \MaxPer^{\frac12}(\MaxAvg_\Ss f)$. Therefore, we have by Lemma \ref{lem:levelSetMaxPer}
  \begin{align*}
    \int_\Omega \MaxAvg f \dd\mu & \leq 2  \int_\Omega \MaxPer^{\frac12}(\MaxAvg_\Ss f) \dd\mu  \leq 4\int_\Omega \MaxAvg_\Ss f \dd \mu 
    \leq   4\int_\Omega \SpOp_\Ss f \dd \mu.
  \end{align*}
Conversely, for each sparse sequence $\{S_k\}_{k\geq 0}$ we compute
  \begin{align*}
    \left\| \SpOp_\Ss f \right\|_{L^1(\Omega)} &=  \sum_{k\geq 0} \int_{S_k} |\cexp_k [f]| \dd\mu \\
    & \leq 2 \sum_k \int_{S_k} |\cexp_k [f]| \cdot  \cexp_k [\ind_{S_{\geq k+1}^c}] \dd\mu \\
    & \leq 2 \int_\Omega  \MaxAvg f \cdot \sum_{k\geq 0} \ind_{S_k \cap S_{\geq k+1}^c}\dd\mu \leq 2\,\|\MaxAvg f\|_{L^1(\Omega)},
  \end{align*}
  where we used the sparsity condition $\cexp_k[\ind_{S_{\geq k+1}^c}] \geq\frac12$ on $S_k$ in the first inequality, and the fact that the sets $S_k \cap S_{\geq k+1}^c$ are pairwise disjoint in the last step.
  This finishes the proof.
\end{proof}

We remark that the percentile enlargement $\MaxPer$ in Proposition \ref{prop:keyIneqMaximals} is necessary: in Section \ref{sec:sec2} we will show that the naive estimate $\MaxAvg f \lesssim \SpOp_\Ss f$ cannot hold for any sparse $\Ss$ with a universal constant.

We next turn to the proof of Theorem \ref{th:MaxStrong}. To that end, we generalize Proposition \ref{prop:keyIneqMaximals} to general maximal functions. The setting is the following: given a countable collection of sets $\mathcal{E}$ on any measure space $(X,\mu)$ such that $0<\mu(R) < \infty$ for all $R\in\mathcal{E}$, we define  
$$
\MaxAvg_{\mathcal{E}} f  = \sup_{R \in \mathcal{E}} |\langle f \rangle_R| \ind_R, \quad \MaxPer_\mathcal{E}^rf:=\sup_{R\in \mathcal{E}} \per_R^r(|f|) \ind_R.
$$
\begin{theorem}\label{rem:SparseDomGeneralSets}
Let $\mathcal{E}$ be a countable collection of sets of finite measure. For each $f \in L^\infty(X)$, there exists a sparse subfamily $\Ss \subseteq \mathcal{E}$ such that
$$
\MaxAvg_{\mathcal{E}} f \leq 2\,\MaxPer_\mathcal{E}^{\frac12}(\MaxAvg_{\Ss}f).
$$   
\end{theorem}

\begin{proof}
We can assume without loss of generality that $\|f\|_{L^\infty(X)} = 1$. For every integer $m \geq 1$ define the collections
\begin{align*}
\mathcal{E}_m = \cbraceb{ R \in \mathcal{E}:\, |\langle f \rangle_R| \in (2^{-m}, 2^{-m+1}] }.
\end{align*}
We will inductively produce a sequence of subcollections $\mathcal{F}_m \subseteq \mathcal{E}_m$ which,
when taken cumulatively,
sparsely dominate each layer $\MaxAvg_{\mathcal{E}_m}f$ after applying the maximal median.

Let $m \geq 1$ and assume inductively that we have constructed $\mathcal{F}_{m-1}$ (taking $\mathcal{F}_0 = \varnothing$).
First, we may enumerate the elements of $\mathcal{E}_m$ in any arbitrary way, we shall denote this enumeration by $\mathcal{E}_m = \{R_{m,1}, R_{m,2}, R_{m,3}, \dots\}$.

For each integer $i \geq 1$ we will, inductively, decide whether $R_{m,i}$ belongs in the collection $\mathcal{F}_m$ or not.
In order to make the selection process rigorous, we will define a function $\varphi : \mathbb{N} \times \mathbb{N} \to \{\text{True}, \text{False}\}$ and set $\mathcal{F}_m = \{R_{m,i}:\, \varphi(m,i) = \text{True}\}$. It will be convenient to introduce some notation for the union of all sets that have been included up to a certain point:
\begin{align*}
H_{m,i} = \bigcup \{ R_{n,j}:\, (n,j) \leq (m,i) \text{ and } \varphi(n,j) = \text{True} \},
\end{align*}
where $(n,j) \leq (m,i)$ denotes the lexicographical order: $(m, i) \leq (m+1,j)$ for all $i,j$ and $(m,i) \leq (m,i+1)$  for all $i$.

We can now define $\varphi$ as follows:
\begin{align*}
\varphi(m,i) = \text{True} \iff \frac{\mu(R_{m,i} \cap H_{m,{i-1}} )}{\mu(R_{m,i})} \leq \frac{1}{2}.
\end{align*}
In this way, $\varphi(m,i)$ is well-defined for every pair of integers $(m,i) \in \N \times \N$.

It is easy to check that the produced family is sparse. Indeed, we can define
\[
E_{R_{m,i}} = \begin{cases}
  R_{m,i} \setminus H_{m, i-1}, & \text{if } i>1, \\
  R_{m,1} \setminus \bigcup_{j} H_{m-1,j}, & \text{if } m>1 \text{ and } i=1,\\
  R_{1,1}, & \text{if } m=1 \text{ and } i=1.
\end{cases}
\]
Our selection condition guarantees that 
$$
\mu(E_{R_{m,i}}) \geq \frac{1}{2} \mu(R_{m,i}),
$$
and it is disjoint from all the previously selected sets. In order to show that we have the desired domination, fix any $R = R_{m,i} \in \mathcal{E}_m$. We can assume $R \in \mathcal{E}_m$ for \emph{some} $m$,
since otherwise $\langle f \rangle_R = 0$ and there is nothing to dominate. If $R_{m,i}$ was selected, then $\MaxAvg_{\Ss} f \geq |\langle f \rangle_R|$ on $R$, and thus $|\langle f \rangle_R| \leq \MaxPer_{\mathcal{E}}^{\frac{1}{2}} \MaxAvg_{\Ss} f$.
Suppose, instead, that $R$ was \emph{not} selected. Then
\begin{align*}
    \frac{\mu(R_{m,i} \setminus E_{R_{m,i}})}{\mu(R_{m,i})} > \frac{1}{2}.
\end{align*}
By construction, on $R_{m,i} \setminus E_{R_{m,i}}$ we have $\MaxAvg_{\Ss}f \geq 2^{-m}$, so the above inequality guarantees
\begin{align*}
\per_R^{\frac{1}{2}}(\MaxAvg_{\Ss}f) \geq 2^{-m},
\end{align*}
and so
\begin{align*}
|\langle f \rangle_R| \leq 2^{-m+1} \leq 2 \,\MaxPer^{\frac{1}{2}}_{\mathcal{E}}(\MaxAvg_{\Ss} f)
\end{align*}
on $R$.
\end{proof}

Denote the standard dyadic system on $\R^d$ by $\D(\R^d)$. If, for $d_1,d_2 \geq 1$, we take
$$
\mathcal{E} = \D(\R^{d_1}) \times \D(\R^{d_2}),
$$
we obtain the statement of Theorem \ref{th:MaxStrong}: for each $f \in L^\infty(\R^{d_1+d_2})$, there exists a sparse family $\Ss \subseteq \D(\R^{d_1}) \times \D(\R^{d_2})$ such that   
$$
\MaxAvg_{\D(\R^{d_1}) \times \D(\R^{d_2})} f(x) \leq 2 \,\MaxPer_{\D(\R^{d_1}) \times \D(\R^{d_2})}^{\frac12}(\MaxAvg_\Ss f)(x), \qquad x \in \R^{d_1+d_2}.
$$
As we mentioned in the Introduction, this domination result sheds new light on the negative result in \cite{BCOR19}. We expect that it will open the door to considering plausible sparse estimates for multiparameter multi-scale operators like the double Hilbert transform. 

\section{Regular martingales - multiscale operators and dyadic harmonic analysis} \label{sec:sec2}

\subsection{Martingale transforms and square functions} In this section, we make the additional standing assumption that the filtration $\{\F_k\}_{k\geq 0}$ is regular. This means that there exists a constant $R>0$ such that for every nonnegative $f \in L^1(\Omega)$ we have the estimate
$$
\cexp_k [f] \leq R \; \cexp_{k-1} [f], \qquad k\geq 1.
$$

In contrast to Section \ref{sec:sec1}, here we use the language of sparse stopping times $\cbrace{\nu_k}_{k\geq 0}$ from \cite{DPS23}, rather than sparse sequences $\cbrace{S_k}_{k\geq 0}$ as in Definition \ref{def:sparsemart}. Since we will now be dealing with multiscale operators, we need to introduce stopped versions of the objects used in Section \ref{sec:sec1}. Let $\nu$ be a stopping time.
Given a martingale $\{f_k\}_{k\geq 0}$ we  define the \emph{stopped} martingale $\{f^\nu_k\}_{k\geq 0}$ as 
$$
f^\nu_k:= \sum_{m=0}^k df_m \ind_{\cbrace{m \leq \nu}}.
$$
Likewise, the stopped conditional median $\per_\nu^r$ is given by
$$
\per_{\nu}^r(f) := \sum_{k\geq 0} \ind_{\{\nu=k\}} \per_k^r(f).
$$
Finally, we define the localized maximal function by
$$
\MaxAvg_{(\nu)} f:= \sup_{k \geq 0} \ind_{\cbrace{k\geq \nu}}|\cexp_k [f]|.
$$
Let $\{\sigma_k\}_{k\geq 0}$ be a predictable, $L^\infty$-normalized sequence of functions, i.e. assume that $\sigma_k$ is $\F_{k-1}$-measurable for all $k\geq 1$ and $\sigma_0$ is  $\F_0$-measurable and suppose
$$
\sup_{k\geq 0} \|\sigma_k\|_{L^\infty(\Omega)} \leq 1.
$$
The martingale transform $T=T(\sigma)$ associated with such a sequence and its associated maximal truncation $T^*$ are the operators respectively given by
\begin{align*}
    T f := \sum_{k \geq 0} \sigma_k df_k, \qquad\qquad T^* f := \sup_{\ell \geq 0}\, \absB{\sum_{k=0}^\ell \sigma_k df_k }.
\end{align*}
The corresponding objects localized at a stopping time $\nu$ are defined as follows:
\begin{align*}
    T_{(\nu)} f& :=  \ind_{\{\nu<\infty\}} \sigma_\nu\cexp_{\nu}[f] + \sum_{k \geq 0} \ind_{\cbrace{k>\nu}} \sigma_k df_k, \\
    \quad T_{(\nu)}^*f & := \sup_{\ell \geq 0} \,\absB{ \ind_{\cbrace{\ell\geq \nu}}\sigma_{\nu}\cexp_{\nu}[f] + \sum_{k= 0}^\ell \ind_{\cbrace{k>\nu}} \sigma_k df_k}.
\end{align*}
With this notation, $T_{(\nu)}f$ is slightly different from the martingale transform $Tf$ \emph{started} at $\nu$, usually denoted by $T_\nu f$. We call $T_{(\nu)}f$ the localized operator, which is natural when the operator only sees the future of $\nu$.
The started martingale transform $T_\nu f$ corresponds to the complement of the operator stopped at $\nu$, which we will not use. Finally, we define the square function and its localization at a stopping time $\nu$ by
\begin{align*}
  Sf := \brB{\sum_{k \geq 0} (df_k)^2}^\frac12, \qquad S_{(\nu)} f = \brB{\ind_{\{\nu<\infty\}}|\cexp_{\nu}[f]|^2 + \sum_{k\geq 0} \ind_{\cbrace{k > \nu}} (df_k)^2}^\frac12.
\end{align*}

We will study the conditional medians of the martingale transform $T(\sigma)$ and the square function $S$ simultaneously.

\begin{proposition} \label{thm:median_of_mt_and_sf}
Let $T=T(\sigma)$ be a martingale transform and let $\nu$ be a stopping time. Then for any $f \in L^\infty(\Omega)$ and $r \leq \frac{1}{2(R+2)}$ we have
\begin{align*}
    \per_\nu^{(R+2)r}(T_{(\nu)}^* f) &\leq 2 r^{-\frac12}\,\per_\nu^r(\MaxAvg_{(\nu)} f),\\
    \per_\nu^{(R+2)r}(S_{(\nu)} f) &\leq 2 r^{-\frac12}\,\per_\nu^r(\MaxAvg_{(\nu)} f).
\end{align*}
\end{proposition}
\begin{proof}
  Let $T$ be either the maximally truncated martingale transform $T^*$ or the square function $S$. It is enough to prove for all $A \in\mc{F}_k$
  $$
  \per_k^{(R+2)r}(T_{(k)}(f\ind_A)) \leq  2 r^{-\frac12}\,\per_k^r(\MaxAvg_{(k)} (f\ind_A)),
  $$
  and we may further assume that $A \cap \{\per_k^r(\MaxAvg_{(k)} (\ind_Af))=0\}=\emptyset$, because otherwise both sides of the inequality are $0$. Consider the set
  \begin{align*}
    B = \{x \in A:\, \MaxAvg_{(k)} f(x) > \per_k^r(\MaxAvg_{(k)} f)(x)\}.
  \end{align*}
  By definition, $\mu(B) \leq r \mu(A)$ because $A$ is $k$-measurable. Consider the enlargement
  \begin{align*}
    \widetilde{B} = \Big\{x \in \Omega:\, \MaxAvg_{(k)}(\ind_B) > \frac{1}{R+1} \Big\}.
  \end{align*}
  By Doob's maximal theorem, we have $\mu(\widetilde{B}) \leq (R+1)r\mu(A)$. Define the following stopping time:
  \begin{align*}
    \tau(x) = \inf\Big\{ m \geq k:\, \cexp_m[\ind_B](x) > \frac{1}{R+1} \Big\}, \qquad x \in \Omega.
  \end{align*}
  We claim that $\abs{f^\tau} \leq \per_k^r(\MaxAvg_{(k)} f)$. Indeed, fix $x \in B$ and let $m \leq \tau(x)$. Then
  \begin{align*}
    \cexp_{m-1}[\ind_B](x) \leq \frac{1}{R+1} \implies \cexp_m[\ind_B](x) \leq \frac{R}{R+1} < 1,
  \end{align*}
  by the regularity of the filtration. On the other hand, if $|f_m(x)| > \per_k^r(\MaxAvg_{(k)} f)(x)$, there exists $A_m \in \F_m$ such that $x\in A_m \subseteq B$, and therefore
  $$
  \cexp_m[\ind_B](x) = 1\quad \mbox{or} \quad \mu(A_m)=0.
  $$ 
  But $\cexp_{m-1}[\ind_B](x) < 1$, so we must have $\abs{f_m(x)} \leq \per_k^r(\MaxAvg_{(k)} f)(x)$.
  Next, note that on the complement of $\widetilde{B}$ we have
  \begin{align*}
    T_{(k)}(f\ind_A)_n = T_{(k)}(f^{\tau}\ind_A)_n,
  \end{align*}
  thus
  \begin{align*}
    \mu(|T_{(k)}(f\ind_A)| >  2 r^{-\frac12}\,\per_k^r(\MaxAvg_{(k)} f)) &\leq \mu\br{\widetilde{B}} + \mu\brb{|T_{(k)}(f^\tau\ind_A)| >  2 r^{-\frac12}\cdot \per_k^r(\MaxAvg_{(k)} f)}.
  \end{align*}
  If $T$ is the maximal truncation of a martingale transform, the above implies
  \begin{align*}
  \mu\brb{|T_{(k)}(f\ind_A)| & >  2 r^{-\frac12}\per_k^r(\MaxAvg_{(k)} f)}\\
  & \leq (R+1)r\mu(A) + \mu\brb{|\MaxAvg_{(k)}(T_{(k)}(f^\tau\ind_A))| >  2 r^{-\frac12}\cdot \per_k^r(\MaxAvg_{(k)} f)}.    
  \end{align*}
To estimate the second term on the right-hand side above, we use Chebyshev's conditional inequality, the measurability of $\per_k^r(\MaxAvg_{(k)} f)$, and the $L^2$-boundedness of $T_{(k)}$ and Doob's maximal function:
  \begin{align*}
    \mu\brb{\MaxAvg_{(k)}(T_{(k)}(f^\tau\ind_A)) > &  2 r^{-\frac12}\per_k^r(\MaxAvg_{(k)} f)} \\
    & = \int_\Omega \mu\Big(\MaxAvg_{(k)}(T_{(k)}(f^\tau\ind_A)) >  2 r^{-\frac12}\per_k^r(\MaxAvg_{(k)} f) \big| \F_k\Big) \dd\mu \\
    & \leq \frac{r}{4} \int_\Omega \per_k^r(\MaxAvg_{(k)} f)^{-2} \cdot \MaxAvg_{(k)}(T_{(k)}(f^\tau\ind_A))^2 \dd\mu \\
    & = \frac{r}{4} \int_\Omega  \MaxAvg_{(k)}\left(T_{(k)}\left(f^\tau \cdot \per_k^r(\MaxAvg_{(k)} f)^{-1}\cdot \ind_A\right)\right)^2 \dd\mu \\
    & \leq r\left\|f^\tau \cdot \per_k^r(\MaxAvg_{(k)} f)^{-1} \cdot \ind_A\right\|_{L^2(\Omega)}^2 \leq r\mu(A),
  \end{align*}
  which is enough to conclude. If instead $T$ is the square function, a similar computation holds using the equality
  $$
  \int_\Omega g (S_{(k)}f)^2 \dd\mu = \int_\Omega g|f|^2  \dd\mu,
  $$
  if $g$ is $\F_k$-measurable. 
\end{proof}

As a consequence of Proposition \ref{thm:median_of_mt_and_sf}, we obtain Theorem \ref{th:theoremA} for martingale transforms. We include the precise statement below. 

\begin{corollary}\label{cor:thmAMT}
  Let $r := 1/(2(R+3))$ and let $T = T(\sigma)$ be a martingale transform.
  For every $f \in L^\infty(\Omega)$ there exists a sparse sequence of stopping times $\{\nu_j\}_{j\geq 0}$ such that
  \begin{align*}
    T^*f \lesssim \sum_{j=0}^\infty \per_{\nu_j}^r(\MaxAvg_{(\nu_j)} f) \ind_{\{\nu_j < \infty\}}.
  \end{align*}
\end{corollary}
\begin{proof}[Proof of Theorem \ref{th:theoremA} for martingale transforms]
  Set $\nu_0=0$ and consider the stopping time
  \begin{align*}
    \nu_1 = \inf \cbraceb{m \geq 0:\, \max((T^*f)_m, |f_m|) >  2 r^{-\frac12}\per_0^r(\MaxAvg f)}
  \end{align*}
  and set $B = \{\nu_1 < \infty\}$. By Proposition \ref{thm:median_of_mt_and_sf}, we have
  \begin{align*}
    \mu(B) &\leq \mu(\MaxAvg f > \per_0^r(\MaxAvg f)) + \mu(T^* f >  2 r^{-\frac12}\per_0^r(\MaxAvg f)) \\
    &\leq r + (R+2)r = (R+3)r \leq \tfrac{1}{2}.
  \end{align*}
If $\mu(B) = 0$ we are done. Otherwise, we can write
  \begin{align*}
    |(Tf)_n^*| \leq \underbrace{\big|(T^*f)_{\nu_1-1 \land n}\big| +|f_{\nu_1-1}|\vphantom{ \sum_{\nu_1+1 \leq m \leq n}}}_{\mathrm{I}} + \Big|\underbrace{\sigma_{\nu_1} f_{\nu_1} + \sum_{\nu_1+1 \leq m \leq n} \sigma_{m} df_m}_{\mathrm{II}}\Big|,
  \end{align*}
  almost everywhere in $B$. By definition, $\mathrm{I} \leq 2r^{-\frac12}\per_0^r(\MaxAvg f)$, and so we can just focus on $\mathrm{II}$, for which we adapt the idea from \cite{Lac2017} in the dyadic setting: the term $\mathrm{II}$ is a martingale transform with respect to the probability space started at $\nu_1$. To be precise, let $\Omega_k = \{\nu_1 = k\}$. Then, we can form the probability space restricted to $\Omega_k$ in the usual way. Let $\widetilde{\mu}$ be the probability measure conditioned to $\Omega_k$ and let
  \begin{align*}
    \mathcal{G}_m = \F_{m+k}.
  \end{align*}
  Then $(\Omega_k, (\mathcal{G}_m)_{m \geq 0}, \widetilde{\mu})$ is a filtered probability space and
  \begin{align*}
    \widetilde{\sigma}_m &:= \sigma_{m+k} \\
    \widetilde{f}_m &:= f_{m+k}
  \end{align*}
  are, respectively, a predictable sequence and a martingale. Also,
  \begin{align*}
    d\widetilde{f}_m = \begin{cases}
        f_k &\text{if } m = 0 \\
        df_{m+k} &\text{if } m \geq 1,
    \end{cases}
  \end{align*}
  which explains the first term in $\mathrm{II}$. Therefore, almost everywhere on $\Omega_k$ we have
  \begin{align*}
    \mathrm{II} = \widetilde{T}\widetilde{f}_n,
  \end{align*}
  where $\widetilde{T}$ is the martingale transform associated to $\widetilde{\sigma}$.  
  Hence, we can iterate this term. At the end of the process, we obtain a sparse sequence of stopping times, which yields the result.
\end{proof}

The corresponding result for $S$ is proved in a similar manner.

\begin{corollary} \label{cor:thmASF}
 Let $r := 1/(2(R+3))$. For every $f \in L^\infty(\Omega)$ there exists a sparse sequence of stopping times $\{\nu_j\}_{j\geq 0}$ such that
  \begin{align*}
    Sf \lesssim \brB{\sum_{j=0}^\infty [\per_{\nu_j}^r(\MaxAvg_{(\nu_j)} f)]^2 \ind_{\{\nu_j < \infty\}}}^\frac12.
  \end{align*}
\end{corollary}

\begin{proof}[Proof of Theorem \ref{th:theoremA} for square functions]
 Set $\nu_0=0$ and define $\nu_1$ in a similar way as before
   \begin{align*}
    \nu_1 = \inf \cbraceb{m \geq 0:\, \max((Sf)_m, |f_m|) >  2 r^{-\frac12}\per_0^r(\MaxAvg f)}.
  \end{align*}
  Setting $B=\{\nu_1<\infty\}$, we have again $\mu(B) \leq \frac12$ by Proposition \ref{thm:median_of_mt_and_sf}. For points in $B$, this time we split
  \begin{align*}
  (Sf)^2 & \leq \sum_{m < \nu_1} |df_m|^2 + |df_{\nu_1}|^2 + \sum_{m > \nu_1} |df_m|^2 \\
  & \leq \sum_{m < \nu_1} |df_m|^2 + 2|\cexp_{\nu_1-1}[f]|^2 + 2|\cexp_{\nu_1}[f]|^2 + \sum_{m > \nu_1} |df_m|^2 \\
  & \lesssim [\per_0^r(\MaxAvg f)]^2 + [\per_{\nu_1}^r(\MaxAvg_{\nu_1} f)]^2+(S_{(\nu_1)}f)^2, 
  \end{align*}
  and the rightmost term above can be iterated as in the proof of Corollary \ref{cor:thmAMT}.
\end{proof}

Clearly, Corollary \ref{cor:thmASF} recovers the classical sparse domination result for the dyadic square function, which states that for every $f \in L^\infty(\Omega)$ there is a sparse sequence of stopping times $\{\nu_j\}_{j\geq 0}$ (or equivalently a sparse sequence  $\cbrace{S_k}_{k\geq 0}$) such that
  \begin{align*}
    (Sf)^2 \lesssim \sum_{j=0}^\infty \cexp_{\nu_j}[|f|]^2 \ind_{\{\nu_j < \infty\}}= \sum_{k = 0}^\infty \ind_{S_k} \cexp_k[|f|]^2.
  \end{align*}
  
The $\Hardy^p(\Omega) \to L^p(\Omega)$ results for martingale transforms, $0<p<\infty$, and the estimate of $\|Sf\|_{L^p(\Omega)}$ by $\|\MaxAvg f\|_{L^p(\Omega)}$ 
follow from the fact that the cancellative sparse operator
$$
\CancSpOp_\Ss f := \sum_{k\geq 0} \ind_{S_k}\per_k^r(|f|) . 
$$
is  $L^p$-bounded for all $p \in (0,\infty)$. 

\begin{lemma}\label{lem:LpboundCancellativeSparseOp} Let $0<p<\infty$ and let $\Ss$ be a sparse sequence. For all $f \in L^p(\Omega)$  we have
$$
\|\CancSpOp_\Ss f \|_{L^p(\Omega)} \lesssim \| f \|_{L^p(\Omega)}.
$$
\end{lemma}

\begin{proof}
We can assume $f \geq 0$. For all $0<q<\infty$ and each $A\in \F_k$, we have
\begin{align*}
\per_k^{r}(f)^q & = \frac{\mu(A)\per_k^{r}(f)^q}{\mu(A)}\\
& = \frac{\per_k^{r}(f)^q}{\mu(A)}\left[\mu(\{x\in A: f(x) \geq \per_k^{r}(f)\}) + \mu(\{x\in A:f(x) < \per_k^{r}(f)\})\right] \\
& \leq \frac{1}{\mu(A)} \left[\int_A f^q d\mu + (1-r)\mu(A)\per_k^{r}(f)^q\right],
\end{align*}
from which we obtain
$$
\per_k^{r}(f) \leq \left(\frac{1}{r\mu(A)} \int_A f^q \dd\mu\right)^{\frac{1}{q}},
$$
for all $A \in \F_k$, which implies $\per_k^{r}(f) \leq r^{-1/q} (\cexp_k[f^q])^{1/q}$. 
Choosing $q=p/2$ and using the boundedness of $\MaxAvg$ on $L^2(\Omega)$, we get
    \begin{align*}
        \|\CancSpOp_{\Ss} f \|_{L^p(\Omega)}^p&\leq \sum_{k\geq 0} \int_{S_k} \left(\per_k^{r}(f)\right)^p \dd \mu\\
        & \leq r^{-2} \sum_{k\geq 0} \int_{S_k} \brb{\cexp_k[f^{\frac{p}{2}}]}^2 d \mu\\ 
        & \leq 4r^{-2} \sum_{k\geq 0} \int_{S_k} \brb{\cexp_k[f^{\frac{p}{2}}]}^2 \cexp_k[\ind_{S_{\geq k+1}^c}] \dd \mu\\ 
        & = 4r^{-2} \sum_{k\geq 0} \int_{S_k \cap S_{\geq k+1}^c} \brb{\cexp_k[f^{\frac{p}{2}}]}^2  \dd \mu\\ 
        & \leq 4r^{-2} \sum_{k\geq 0} \int_{S_k \cap S_{\geq k+1}^c} \brb{\MaxAvg\brb{f^{\frac{p}{2}}}}^2  \dd \mu\\
        & \leq 4r^{-2} \int_{\Omega} \brb{\MaxAvg\brb{f^{\frac{p}{2}}}}^2  \dd \mu \leq \|f\|_{L^p(\Omega)}^p,   
    \end{align*}
    as desired.
\end{proof}

\begin{remark}
The proofs of Corollaries \ref{cor:thmAMT} and \ref{cor:thmASF} use the regularity of the filtration in an explicit way. One may wonder if regularity is necessary, since boundedness results near the $L^1$-endpoint for martingale transforms hold for general filtered spaces, without the regularity condition. The same is true for the positive sparse domination result in \cite{DPS23}, which is valid for general filtrations. However, \cite[Example 8.1]{BG1970} shows that regularity is indeed necessary for our sparse domination results to hold. Otherwise, Theorem \ref{th:theoremA} and Lemma \ref{lem:LpboundCancellativeSparseOp} could be used to obtain the inequality
$$
\|Sf\|_p \lesssim \|\MaxAvg f\|_p
$$ 
for $0<p<1$ in a general probability space, which is shown to be false in the above-mentioned reference.    
\end{remark}

\label{sec:21}

\subsection{Dyadic harmonic analysis: Haar shifts} In this subsection we restrict ourselves to the filtration generated by the usual dyadic system $\D=\{\D_k\}_{k\in \Z}$ in $\R^d$. This means that we have a two-sided filtration that induces the family of conditional expectations
$$
\cexp_k [f] = \sum_{Q \in \D_k} \langle f \rangle_Q \ind_Q, \quad k \in \Z.
$$
We assume that the underlying measure is the Lebesgue measure, although any dyadically doubling one would yield exactly the same results. A measure $\mu$ is dyadically doubling if $\mu(\widehat{Q}) \leq C_{\mathrm{doub}}(\mu) \mu(Q)$ for all $Q\in \D$, which implies that the filtration is regular with regularity constant $R=C_{\mathrm{doub}}(\mu)$. Of course, the Lebesgue measure $m$ is dyadically doubling with $C_{\mathrm{doub}}(m)=2^d$. The fact that the filtration is two-sided changes nothing in practice, since the $\sigma$-algebra generated by $\D_{-\infty}$ is almost trivial: indeed,
$\cap_{k\in \mathbb{Z}} \sigma(\D_k)$ has finitely many elements, all of which have infinite measure. For a careful analysis of what happens in the general two-sided context, we refer to \cite{Treil2013}. We will avoid the related technicalities here. In the dyadic setting, Doob's maximal function can be written as
$$
\MaxAvg_\D f(x) = \sup_{x\in Q \in \D} \left|\langle f \rangle_Q \right|, \qquad x \in \R^d.
$$
We denote by $\D(Q)$ the family of dyadic cubes contained in $Q\in\D$, and for each $k>0$ we denote by $\D_k(Q)$ the family of dyadic subcubes of $Q$ with side length equal to $2^{-k}\ell(Q)$. Let $\{h_Q\}_{Q\in \mathscr{D}}$ be a fixed Haar basis of $L^2(\R^d)$. A Haar shift is an operator of the form
$$
\Sh f(x) := \sum_{Q\in \D} \sum_{T\in \D_t(Q)} \sum_{S \in \D_s(Q)} \alpha_{TS}^Q \langle f, h_T\rangle h_S(x),\qquad x \in \R^d,
$$
where $t$ and $s$ are nonnegative integers and the coefficients $\alpha_{TS}^Q$ satisfy   $$\|\alpha\|_{\ell^\infty} :=\sup_{Q,T,S} |\alpha_{TS}^Q| <\infty,$$ so that $\Sh$ is bounded on $L^2(\R^d)$, and in fact $\|\Sh\|_{L^2 \to L^2} \eqsim \|\alpha\|_{\ell^\infty}$. The pair of parameters $(t,s)$ is called the complexity of $\Sh$, which we omit from notation. 
Truncations of $\Sh$ and $\MaxAvg_\D$ are defined in a similar way as above, localizing to a $Q\in\D$. This is analogous to the generation-based truncations of Section \ref{sec:sec1}:
\begin{align*}
    \MaxAvg_{Q} f(x) &:= \sup_{x\in R \in \D(Q)} \left|\langle f \rangle_R \right|,&& x \in Q, \\ \Sh_{Q} f(x) &:= \sum_{R \in \D(Q)} \sum_{T\in \D_t(R)} \sum_{S \in \D_s(R)} \alpha_{TS}^R \langle f, h_T\rangle h_S(x) , &&x \in Q.
\end{align*}
Given $Q\in \D$ and an open set $A\subseteq Q$, define $A^{(0)}:=A$, and inductively, 
$$
A^{(n)} := \left\{x \in Q: \MaxAvg_\D(\ind_{A^{(n-1)}}) > \frac{1}{2^d+1}\right\}, \quad n\geq 1.
$$
By the weak $(1,1)$-boundedness of $\MaxAvg_\D$, we have
$$
|A^{(n)}| \leq (2^d+1)^n |A|,
$$
while by construction if $R\subseteq A$ is a cube in $\D$, then $R^{(n)} \subseteq A^{(n)}$. Moreover, if $R \subseteq A^{(n)}$ is a cube in $\D$ which is maximal with respect to inclusion, then any $T\in \D_{n-1}(R)$ satisfies $T \cap A^{c} \not= \emptyset$. As in the Subsection \ref{sec:21}, we will study the maximal truncations
$$
\Sh^* f(x) = \sup_{\ell < m} \,\absB{\sum_{k=\ell}^m \sum_{x\in Q\in \D_k} \sum_{T\in \D_t(Q)} \sum_{S \in \D_s(Q)} \alpha_{TS}^Q \langle f, h_T\rangle h_S(x)},\qquad x \in \R^d.
$$
The localized maximal truncations $\Sh_Q^*$ are defined restricting the second outermost sum above to cubes contained in $Q$. These also satisfy $\|\Sh_Q^*\|_{L^2 \to L^2} \lesssim \|\alpha\|_{\ell^\infty}$ uniformly on $Q\in \D$. The analogue of Proposition \ref{thm:median_of_mt_and_sf} for $\Sh^*$ is the following local estimate.

\begin{proposition} \label{dyadic:median_of_T}
  Let $C_0=(2^d+1)^{s+t+1}+1$. If $r < \frac{1}{C_0}$ and $Q \in \D$, then for all $f \in L^\infty(Q)$ we have
  \begin{align*}
    \per_Q^{C_0r}(\Sh_Q^*f) \leq {r^{-\frac12}} \|\Sh_Q^*\|_{L^2\to L^2} \cdot \per_Q^r(\MaxAvg_Q f).
  \end{align*}
\end{proposition}
\begin{proof}
We may assume, without loss of generality, that $\per_Q^r(\MaxAvg_Q f) = 1$. Set $N=t+s+1$ and define
$$
B=\{x \in Q:\, \MaxAvg_Qf(x) > 1\},
$$
so $|B| \leq r|Q|$. Fix a constant $\lambda>0$ whose value we will choose momentarily. By the definition of $\per_Q^r$ we have
  \begin{align*}
    |\{x \in Q:\, \Sh_Q^*f(x) > \lambda\}| &\leq  |B^{(N)}| + |\{x \in Q\setminus B^{(N)} :\, \Sh_Q^*f(x) > \lambda\}| \\
    &\leq  (2^d+1)^Nr|Q| + |\{x \in Q\setminus B^{(N)} :\, \Sh_Q^*f(x) > \lambda\}|.
  \end{align*}
  To estimate the second term on the last display above, we write
  $$
  B^{(1)} = \bigsqcup_j Q_j,
  $$
  where the $Q_j$ are maximal dyadic cubes inside $B^{(1)}$. We next split $f$ into its bad and good parts
  $$
  f = \bracB{f\ind_{B^{(1)}} - \sum_j \langle f \rangle_{Q_j} \ind_{Q_j}} + \bracB{f\ind_{Q \setminus B^{(1)}} + \sum_j \langle f \rangle_{Q_j} \ind_{Q_j}} =: f_1+f_2.  
  $$
  By construction, $\supp(\Sh_Q f_1) \subset B^{(N)}$, while for each $j$, $|\langle f \rangle_{Q_j}| \leq 1$ and so $|f_2| \leq 1$. Therefore, using the $L^2$-boundedness of $\Sh_Q^*$ and sublinearity we get 
  \begin{align*}
    |\{x \in Q\setminus B^{(N)} :\, \Sh_Q^*f(x) > \lambda\}| &\leq \frac{1}{\lambda^2} \int \Sh_Q^* f_2(x)^2 \dd x \\
    &\leq \frac{\|\Sh_Q^*\|_{L^2\to L^2}^2}{\lambda^2} \int_Q |f_2(x)|^2 \dd x \\
    &\leq |Q|\frac{\|\Sh_Q^*\|_{L^2\to L^2}^2}{\lambda^2}.
  \end{align*}
  We now choose $\lambda ={r^{-\frac12}} {\|\Sh_Q^*\|_{L^2\to L^2}}$ and combine the above computations to get
  \begin{align*}
    |\{x \in Q:\, \Sh_Q^* f(x) > \lambda\}| &\leq |Q|\big((2^d+1)^Nr +r\big) = [(2^d+1)^N+1]r|Q|,
  \end{align*}
which implies the assertion.
\end{proof}

We can now prove Theorem \ref{th:theoremA} for Haar shift operators. We provide the precise statement below.

\begin{theorem}\label{th:thmAHaarShifts}
Let $f\in L^{\infty}(\R^d)$. Then there exists $r>0$ and a sparse family of cubes $\Ss \subseteq \D$ such that
\begin{equation}\label{eq:sparseForHaarShifts}
    |\Sh^* f| \lesssim_{s,t} {\|\Sh_Q^*\|_{L^2\to L^2}} \sum_{Q\in \Ss} \per_Q^r(\MaxAvg_Qf) \ind_Q.
\end{equation}
\end{theorem}

\begin{proof}[Proof (of Theorem \ref{th:theoremA} for Haar shifts)] By standard reductions,  we may assume that $f$ is a mean-zero function and that there exists $\tilde{Q}_0 \in \D$ such that $\supp(f) \subseteq \tilde{Q}_0$. Since $\langle f \rangle_{\tilde{Q}_0}=0$, defining $Q_0=\tilde{Q}_0^{(s+t+1)}$ yields
$$
\Sh^* f(x) = \Sh_{Q_0}^* f(x), \qquad x \in \R^d.
$$
Fix $r=\frac{1}{2(C_0+1)^2}$, where $C_0$ is the constant in the statement of Proposition \ref{dyadic:median_of_T}. We are going to prove the following claim: there exists a constant $C$ such that for each $Q\in \D$, there exists a pairwise disjoint collection $\F(Q)=\{Q_j\}_j\subseteq \D(Q)$ such that $\sum_j|Q_j| \leq \frac{1}{2}|Q|$ and
  \begin{align*}
    \Sh_Q^*f(x) \leq C \, \per_Q^r(\MaxAvg_Qf) + \sum_{j} |\Sh_{Q_j}^* f(x)|, \quad x\in Q.
  \end{align*}
Indeed, set 
$$
\lambda = r^{-\frac12}\sup_{Q\in \D} \left\{\|\Sh_Q^*\|_{L^2\to L^2}\right\} \per_Q^r(\MaxAvg_Qf),
$$ 
and consider the set 
$$
\Omega = \{x \in Q:\, \Sh_Q^*f(x) > \lambda\} \cup \{x \in Q:\, \MaxAvg_Qf(x) > \per_Q^r(\MaxAvg_Qf)\}.
$$ 
By Proposition \ref{dyadic:median_of_T} and the weak $(1,1)$-boundedness of $\MaxAvg_\D$, we have $|\Omega| \leq (C_0+1)r|Q|$, and so
$$
|\Omega^{(N)}|\leq (2^d+1)^N |\Omega| \leq \tfrac12 |Q|.
$$
We next choose $\F=\{Q_j\}_j$ to be the family of maximal cubes $R\in \D(Q)$ such that $R\subseteq \Omega^{(N)}$, getting
$$
\Omega^{(N)}=\bigsqcup_j Q_j.
$$
We now split
  \begin{align*}
    \Sh_Q^* f &\leq \lambda \ind_{Q \setminus \Omega^{(N)}} + \sum_{j} \ind_{Q_j} \Sh_Q^* f \\
    & \leq \lambda + \sum_{j} \ind_{Q_j}\Sh_Q^* f.
  \end{align*}
Fix $j$ and $x\in Q_j$, and assume $Q\in \D_{m_1}$, $Q_j \in \D_{m_2}$, which holds for some $m_2>m_1$. We may write
\begin{align*}
   \Sh_Q^* f(x) & = \sup_{m_1 \leq \ell < m} \,\absB{\sum_{k=\ell}^m \sum_{x\in R\in \D_k} \sum_{T\in \D_t(R)} \sum_{S \in \D_s(R)} \alpha_{TS}^R \langle f, h_T\rangle h_S(x)} \\
   & \leq \sup_{m_1 \leq \ell < m} \,\absB{\sum_{k=\ell}^{\min\{m,m_2\}} \sum_{x\in R\in \D_k} \sum_{T\in \D_t(R)} \sum_{S \in \D_s(R)} \alpha_{TS}^R \langle f, h_T\rangle h_S(x)}
    + \Sh_{Q_j}^* f(x).
\end{align*}
The second term on the last display above is exactly the one we want to iterate, so it only remains to estimate the first. To do so we rewrite it as
\begin{align*}
\sup_{m_1 \leq \ell < m} \,\absB{\sum_{k=\ell}^{\min\{m,m_2\}} \sum_{x\in R\in \D_k} & \sum_{T\in \D_t(R)} \sum_{S \in \D_s(R)} \alpha_{TS}^R \langle f, h_T\rangle h_S(x)} \\
& = \sup_{0\leq m \leq m_2-m_1}\absB{\sum_{\ell=1}^{m} \sum_{T\in \D_t(Q_j^{(\ell)})} \sum_{S \in \D_s(Q_j^{(\ell)})} \alpha_{TS}^{Q_j^{(\ell)}} \langle f, h_T\rangle h_S(x)} \\
& \leq \|\alpha\|_{\ell^\infty}\sum_{\ell=1}^{N+1} \sum_{T\in \D_t(Q_j^{(\ell)})} \sum_{S \in \D_s(Q_j^{(\ell)})} |\langle f, h_T\rangle h_S(x)|\\
& \hspace{0.5cm}+ \sup_{0\leq m \leq m_2-m_1} \absB{\sum_{\ell=N+2}^{m} \sum_{T\in \D_t(Q_j^{(\ell)})} \sum_{S \in \D_s(Q_j^{(\ell)})} \alpha_{TS}^{Q_j^{(\ell)}} \langle f, h_T\rangle h_S(x)},
\end{align*}
with the last sum being possibly empty. For each $R\in\D$ and $T\in\D_t(R)$, $S\in \D_s(R)$ we have
$$
|\langle f, h_T\rangle h_S| \lesssim_{s,t} \sup_{P\in \D_1(T)} |\langle f \rangle_P|.
$$
On the other hand, if $T\in \D_t(Q_j^{(\ell)})$ for $\ell>0$, then the construction of $\Omega^{(N)}$ implies that $|\langle f \rangle_P| \leq \per_Q^r(\MaxAvg_Qf)$ for all $P\in \D_1(T)$. Therefore, the first term on the right-hand side on the last display above can be estimated by a constant depending on $t,s$ and $d$ times $\per_Q^r(\MaxAvg_Qf)$. For the remaining term, we observe that the sum
$$
\sum_{\ell=N+2}^{m} \sum_{T\in \D_t(Q^{(\ell)}_j)} \sum_{S \in \D_s(Q^{(\ell)})} \alpha_{TS}^{Q^{(\ell)}} \langle f, h_T\rangle h_S
$$
is constant over $Q_j^{(1)}$, and so by maximality of $Q_j$ there must exist $y\in Q_j^{(1)} \setminus \Omega^{(N)}$ and so
\begin{align*}
\absB{\sum_{\ell=N+2}^{m} \sum_{T\in \D_t(Q_j^{(\ell)})} & \sum_{S \in \D_s(Q_j^{(\ell)})} \alpha_{TS}^{Q_j^{(\ell)}} \langle f, h_T\rangle h_S(x)} \\
&= \absB{\sum_{\ell=N+2}^{m} \sum_{T\in \D_t(Q_j^{(\ell)})} \sum_{S \in \D_s(Q_j^{(\ell)})} \alpha_{TS}^{Q_j^{(\ell)}} \langle f, h_T\rangle h_S(y)} 
 \leq |\Sh_Q^*(y)| \leq \lambda.
\end{align*}
This finishes the proof of the claim. Finally, we construct the sparse family as follows: initialize $\Ss=\cbrace{Q_0}$, and apply the claim to $Q_0$ and add the cubes in $\F(Q_0)$ to $\Ss$. Now apply the claim to each $Q\in \F(Q_0)$ and repeat the process inductively. The family $\Ss$ is sparse by construction, and the assertion follows.
\end{proof}

\begin{remark}
If one tracks the constant $C=C(t,s)$ in Theorem \ref{th:thmAHaarShifts}, one  sees that $C \gtrsim 2^{2d(s+t)}$, which is much larger than the optimal $s+t$ that is available via non-cancellative sparse domination, see for example \cite{CR16}. This means that our Theorem \ref{th:thmAHaarShifts} cannot be used as a tool to obtain a cancellative sparse domination result for Calder\'on-Zygmund operators via the dyadic representation theorem \cite{Hy12,Pe07}. In Section \ref{sec:sec3} we will therefore study this case separately. It remains as an interesting open question whether one can significantly improve the dependence on $(t,s)$ in the statement of Theorem \ref{th:theoremA}.
\end{remark}

\subsection{Failure of cancellative sparse domination without enlargement} In this subsection we work on the unit interval $I = [0,1)$ with the canonical dyadic system $\D$. We will show that the inequality 
\begin{equation*}
  \MaxAvg_\D f \lesssim \sum_{J \in \Ss} |\langle f \rangle_J| \ind_J
\end{equation*}
cannot hold in general for a sparse collection of cubes $\Ss \subseteq \D$. Of course, this immediately disproves similar estimates for any multiscale operator whose behavior is similar to that of $\MaxAvg_\D$, like a martingale transform or $\Sh$. For $J \in \D(I)$, let $\varphi_J$ be the affine bijection $J \to I$. One can make this unique by imposing that $\varphi_J$ is orientation-preserving. For convenience, we will write
\begin{align*}
  f_J = (f \circ \varphi_J) \ind_J
\end{align*}
so that in particular we have
\begin{align*}
  \langle f_J \rangle_J = \langle f \rangle_{I}.
\end{align*}
For $N \geq 1$ let $\D_N'$ be the subcollection of $\D_N(I)$ obtained by removing
\begin{itemize}
    \item The two intervals contained in $[0, 2^{-N+1})$.
    \item The two intervals contained in $[1-2^{-N+1},1)$.
\end{itemize}
Enumerate the intervals in $\D_{N}([0,1))$ as $\{J_0, J_1, J_2, \dots, J_{2^{N}-1}\}$.

\begin{lemma}\label{lem:auxiliaryCounterexample}
  Let $f\in L^1(I)$. Then, the transformation
  $$
  T_Nf = 2^{N-1}\sum_{n=2}^{2^{N}-1} (-1)^n f_{J_n} +2^{N-1} f_{[1-2^{-N+1},1)}
  $$
  has the following properties:
  \begin{enumerate}[(i)]
    \item\label{it:T1} $|\langle T_Nf \rangle_J| \geq 2^{N-1}|\langle f\rangle_I|$ for all $J \in \D_N'$.
    \item\label{it:T2}  $\langle T_Nf \rangle_{[0, 2^{-k})} = 0$ for all $k \in \{1, 2, \dots, N\}$
    \item\label{it:T3}  $\langle T_Nf \rangle_I = \langle f \rangle_I$.
  \end{enumerate}
\end{lemma}
\begin{proof}
Properties \ref{it:T1} and \ref{it:T2} follow directly from the definition of $T_N$. Furthermore, we have
  \begin{align*}
    \langle T_Nf \rangle_I &= 2^{N-1}\int f_{[1-2^{-N+1},1)} = \langle f_{[1-2^{-N+1},1)} \rangle_{[1-2^{-N+1},1)},
  \end{align*}
since all the other terms cancel out.
\end{proof}

\begin{figure}[H]
  \begin{tikzpicture}[scale=0.5]
    \draw[thick] (0, 0) -- (30, 0);
    \draw[thick] (0, 10) -- (30, 10);
    \draw[thick] (0, 0) -- (0, 10);
    \draw[thick] (30, 0) -- (30, 10);

    \draw[thick] (0, 2) -- (30, 2);
    \draw[thick] (0, 4) -- (30, 4);
    \draw[thick] (0, 6) -- (30, 6);
    \draw[thick] (0, 8) -- (30, 8);

    \draw[thick] (15, 0) -- (15, 8);
    
    \draw[thick] (7.5, 0) -- (7.5, 6);
    \draw[thick] (22.5, 0) -- (22.5, 6);
    
    \draw[thick] (3.75, 0) -- (3.75, 4);
    \draw[thick] (11.25, 0) -- (11.25, 4);
    \draw[thick] (18.75, 0) -- (18.75, 4);
    \draw[thick] (26.25, 0) -- (26.25, 4);
    
    \draw[thick] (1.875, 0) -- (1.875, 2);
    \draw[thick] (5.625, 0) -- (5.625, 2);
    \draw[thick] (9.375, 0) -- (9.375, 2);
    \draw[thick] (13.125, 0) -- (13.125, 2);
    \draw[thick] (16.875, 0) -- (16.875, 2);
    \draw[thick] (20.625, 0) -- (20.625, 2);
    \draw[thick] (24.375, 0) -- (24.375, 2);
    \draw[thick] (28.125, 0) -- (28.125, 2);

    \node at  (15, 9) {$\langle f \rangle$};

    \node at  (7.5, 7) {$0$};
    \node at  (22.5, 7) {$2\langle f \rangle$};

    \node at  (3.75, 5) {$0$};
    \node at  (11.25, 5) {$0$};
    \node at  (18.75, 5) {$0$};
    \node at  (26.25, 5) {$2^2 \langle f \rangle$};

    \node at  (1.875, 3) {$0$};
    \node at  (5.625, 3) {$0$};
    \node at  (9.375, 3) {$0$};
    \node at  (13.125, 3) {$0$};
    \node at  (16.875, 3) {$0$};
    \node at  (20.625, 3) {$0$};
    \node at  (24.375, 3) {$0$};
    \node at  (28.125, 3) {$2^3\langle f \rangle$};
    
    \node at  (0.9375, 1) {$0$};
    \node at  (2.8125, 1) {$0$};
    
    \node at  (4.6875, 1) {$2^3f$};
    \node at  (6.5625, 1) {$-2^3f$};
    \node at  (8.4375, 1) {$2^3f$};
    \node at  (10.3125, 1) {$-2^3f$};
    \node at  (12.1875, 1) {$2^3f$};
    \node at  (14.0625, 1) {$-2^3f$};
    \node at  (15.9375, 1) {$2^3f$};
    \node at  (17.8125, 1) {$-2^3f$};
    \node at  (19.6875, 1) {$2^3f$};
    \node at  (21.5625, 1) {$-2^3f$};
    \node at  (23.4375, 1) {$2^3f$};
    \node at  (25.3125, 1) {$-2^3f$};
  \end{tikzpicture}
  \caption{Construction of $T_4f$.}
\end{figure}

We are now ready to build our counterexample.

\begin{proposition}\label{prop:counterexample}
Fix $C_0>1$. There exists a constant $C_1>0$ and a family $\{f_n\}_{n\geq 0}$ with the following property: if $\Ss_n$ is an $\eta_n$-sparse family such that
\begin{equation}\label{eq:counterexample}
\MaxAvg_\D f_n \leq C_0 \sum_{J \in \Ss_n} |\langle f_n \rangle_J| \ind_J,   
\end{equation}
then $\eta_n\leq C_1n^{-1}$.    
\end{proposition}

\begin{proof}
Pick $N_0:=1+2\ceil{\log_2(C_0)}$ so that $2^{N_0-1}>C_0$. We define $\{f_n\}_{n\geq 0}$ inductively as follows:
\begin{align*}
  f_n = \begin{cases}
    1, &\text{if } n = 0, \\
    T_{nN_0}f_{n-1}, &\text{if } n \geq 1.
  \end{cases}
\end{align*}
$\MaxAvg_\D f_0 =1$ everywhere on $I$ and so $\Ss_0=\{I\}$ is a $1$-sparse family that satisfies \eqref{eq:counterexample}. Take now $n=1$. We have
\begin{align*}
  \MaxAvg_\D f_1(x) = \begin{cases}
    1, &\text{if } 0\leq x < 2^{-N_0+1}, \\
    2^{N_0-1}, &\text{if } 2^{-N_0+1} \leq x < 1.
  \end{cases}
\end{align*}
By Lemma \ref{lem:auxiliaryCounterexample}, $\langle f_1 \rangle_{[0,2^{-k})}=0$ except for $k = 0$, and so $\Ss_1$ must include $I$. By our choice of $N_0$, for \eqref{eq:counterexample} to hold $\Ss_1$ must include a partition of $[2^{-N_0+1},1)$ formed with cubes of side length at most $2^{-N_0}$. Therefore, the Carleson packing constant $\Car(\Ss_1)$ of $\Ss_1$ is at least
$$
\Car(\Ss_1) \geq 1 + (1-2^{-N_0+1}) = 2 - 2^{-N_0+1}.
$$
Therefore, $\Ss_1$ is at most $\eta_1=(2 - 2^{-N_0+1})^{-1}$-sparse by the Carleson packing characterization of sparse families (see \cite{LN15}).

Next, $\MaxAvg_\D f_2$ takes three different values, namely $1$, $2^{N_0-1}$, and 
$$
2^{N_0-1}\cdot 2^{2N_0-1} > C_0(1+2^{N_0-1}).
$$ 
The properties of the operator $T_N$ imply that $\Ss_2$ must then include $I$, $\D_{N_0}'$, $[1-2^{-N+1},1)$, plus an additional partition of a set of Lebesgue measure equal to $$(1-2^{-N_0+1})(1-2^{-2N_0+1})$$ formed with cubes of side length at most $2^{-3N_0}$. The corresponding Carleson constant is now 
$$
\Car(\Ss_2) \geq 1 + (1-2^{-N_0+1}) + (1-2^{-N_0+1})(1-2^{-2N_0+1}).
$$
Iterating the above reasoning, we find
$$
\Car(\Ss_n) \geq \sum_{j=0}^n \prod_{k=1}^j (1-2^{-kN_0+1}) \gtrsim n, 
$$
which yields $\eta_n \lesssim n^{-1}$, as desired.
\end{proof} 

\section{Calder\'on--Zygmund operators - Euclidean harmonic analysis}\label{sec:sec3}
We now turn to cancellative sparse domination in the Euclidean space $\R^d$ for Calder\'on--Zygmund operators in Theorem \ref{th:theoremC}, which will be based on a general sparse domination principle from \cite{LLO21}. Throughout this section, we omit dependence on $d$ in all implicit constants.

We start by recalling a special case of the main result in \cite{LLO21}.
 For every cube $Q \subseteq \R^d$  let $f_Q \colon Q\to {\R}$ be a measurable function. For  $R \in \D(Q)$ let
 $$
f_{Q,R}(x) := f_Q(x)-f_R(x), \qquad x \in R,
 $$
and define the local grand sharp maximal function   $\MaxAvg_Q^{\#}f \colon Q \to \R$ of the family $\cbrace{f_Q}$ by 
\begin{equation}\label{eq:grandmaxtrunc}
\MaxAvg_Q^{\#}f(x):=\sup_{\substack{R\in {\D}(Q):\\x \in R}}\,\esssup_{x',x'' \in R} \, \absb{f_{Q,R}(x')-f_{Q,R}(x'')}.
\end{equation}

\begin{theorem}[{\cite[Theorem 3.2]{LLO21}}] \label{th:theorem32Emiel}
Let $r= 2^{-d-3}$ and for every cube $Q \subseteq \R^d$ let $f_Q\colon Q \to \R$ be measurable. Then, for any cube $Q_0 \subseteq \R^d$, there exists a $\frac12$-sparse family ${\Ss}\subseteq \D(Q_0)$ such that 
$$
|f_{Q_0}(x)|\lesssim \sum_{Q\in {\Ss}}\left[\per_Q^{r}(|f_Q|)+\per_Q^r(\MaxAvg^{\#}_Qf)\right]\ind_Q(x), \qquad x \in Q_0.
$$
\end{theorem}

To prove Theorem \ref{th:theoremC}, we will apply Theorem \ref{th:theorem32Emiel} with \begin{equation}\label{eq:choicefQ}
    f_Q := T(f\psi_{2Q}), 
\end{equation} where $T$ is an $s$-smooth Calder\'on--Zygmund operator, $f \in L^\infty_c(\R^d)$ and $\psi_{2Q}$ is an $L^{\infty}$-normalized smooth bump function supported on $3Q$ with $\psi_{2Q}\equiv 1$ on $2Q$. Let us start by introducing the involved operators.

\begin{definition}\label{def:CZO}
    Let $T \colon L^2(\R^d) \to L^2(\R^d)$ be a bounded linear operator and $s>0$. We call $T$ an $s$-smooth Calder\'on--Zygmund operator if there is a kernel $K \colon \R^d \times \R^d \setminus\cbrace{x=y}$ such that for all $f \in L^\infty_c(\R^d)$ and $x \notin \supp(f)$ we have
$$
Tf(x) = \int_{\R^d}K(x,y)f(y) \dd y
$$
and, writing $s = m+\delta$ with $m\in \N$ and $\delta \in (0,1]$, there is a constant $C_K>0$  such that
\begin{align*}
\abs{\partial^\alpha_y K(x,y)} &\leq C_K \frac{1}{\abs{x-y}^{d+\abs{\alpha}}}, && \abs{\alpha} \leq m,\\
\abs{\partial^\alpha_y K(x,y)-\partial^\alpha_y K(x,y')} &\leq C_K \frac{\abs{y-y'}^\delta}{\abs{x-y}^{d+\abs{\alpha}+\delta}}, && \abs{\alpha} = m,
\end{align*}and, in addition, there is an $\varepsilon>0$ such that
\begin{align*}
\abs{ \partial^\alpha_y K(x,y)- \partial^\alpha_y K(x',y)} &\leq C_K  \frac{ \abs{x-x'}^\varepsilon}{\abs{x-y}^{d+\abs{\alpha}+\varepsilon}} && \abs{\alpha} \leq m,  \\
\abs{\partial^\alpha_y K(x,y)-\partial^\alpha_y K(x',y)-\partial^\alpha_y K(x,y')+\partial^\alpha_y K(x',y')} &\leq C_K \frac{ \abs{y-y'}^\delta \abs{x-x'}^\varepsilon}{\abs{x-y}^{d+\abs{\alpha}+\delta+\varepsilon}}, && \abs{\alpha} = m,
\end{align*}
for $x\neq y$ and $\abs{x-x'},\abs{y-y'} \leq \frac12\abs{x-y}$.
\end{definition}
Note that Definition \ref{def:CZO} is not symmetric in the $x$- and $y$-variables. Indeed, we assume  smoothness of order $s$ in the $y$-variable, whereas we only assume smoothness of arbitrarily small order $\varepsilon>0$ in the $x$-variable.
  
Theorem \ref{th:theoremC} involves a $Q$-localized, smooth maximal operator on the right-hand side, which we define next. 
For $s\geq  0$ and a cube $Q \subseteq \R^d$ we denote by $\mathcal{F}_s(Q)$ the collection of  $s$-smooth, $L^\infty$-normalized bump functions localized in $Q$, that is, writing $s=m+\delta$ with $m \in \N$ and $\delta \in (0,1]$ (or $\delta=0$ when $s=0$),  the class of functions $\varphi \colon \R^d \to \R$ with $\supp (\varphi) \subseteq {Q}$ and such that for all $y,y' \in \R^d$
\begin{align*}
\abs{\partial^\alpha\varphi(y)} &\leq \frac{1}{\ell(Q)^{\abs{\alpha}}}, &&\abs{\alpha} \leq m,\\
\abs{\partial^\alpha\varphi(y)-\partial^\alpha\varphi(y')} &\leq \frac{\abs{y-y'}^{\delta}}{\ell(Q)^{\abs{\alpha}+\delta}}, &&  \abs{\alpha} =m.
\end{align*}
    Note that for $\varphi \in \mc{F}_s(Q)$ and $\psi \in \mc{F}_s(R)$ with $R\subseteq Q$, we have that $c_{s}\cdot\varphi\cdot \psi \in \mc{F}_s(R)$ for some constant $c_s>0$. 

\begin{definition}
    Let $s\geq 0$ and for $f \in L^1_{\loc}(\R^d)$ define the smooth maximal operator $\MaxAvg^sf$  by
\begin{equation*}
    \MaxAvg^{s}f(x):= \sup_{Q} \,\sup_{\varphi \in \mathcal{F}_s(Q)} \,\frac{1}{\abs{Q}}\absB{\int_{{Q}} f(y)\varphi(y)\dd y}\ind_Q(x),\qquad x \in \R^d,
\end{equation*}
where the supremum is taken over all cubes $Q \subseteq \R^d$ with sides parallel to the coordinate axes. For a cube $Q_0\subseteq \R^d$ define the local version $\MaxAvg_{Q_0}^sf$ by
\begin{equation*}
    \MaxAvg_{Q_0}^sf(x):= \sup_{Q\subseteq Q_0} \,\sup_{\varphi \in \mathcal{F}_s(Q)} \,\frac{1}{\abs{Q}}\absB{\int_{{Q}} f(y)\varphi(y)\dd y} \ind_Q(x),\qquad x \in \R^d.
\end{equation*}
\end{definition}
Let us make a few remarks on the definition of this smooth maximal operator.
\begin{remark}~\label{rem:defMs}
    \begin{enumerate}[(i)]
        \item For $0\leq s_0<s_1$ we have
$
  \MaxAvg^{s_1}f \lesssim_{s_0,s_1} \MaxAvg^{s_0}f.
$
\item  When $s=0$, the extremizer in the definition of $\MaxAvg^0f$ is $\varphi = \mathrm{sgn}\br{f} \ind_Q$, recovering the classical (non-cancellative) Hardy--Littlewood maximal operator, i.e.,
$$
\MaxAvg^0f(x) = \sup_{Q} \,\ip{|f|}_Q\ind_Q,\qquad x \in \R^d.
$$

\item\label{it:defMsiii} Since we can approximate an $s_1$-smooth function by a $C^\infty_c$-function with convergence in $C^{s_0}(\R^d)$ for $0< s_0<s_1$, we have for $x \in \R^d$
$$
  \MaxAvg^{s_1}f(x) \lesssim_{s_0,s_1} \sup_{Q} \,\sup_{\varphi \in \mathcal{F}_{s_0}(Q)\cap C_c^\infty(\R^d)} \,\frac{1}{\abs{Q}}\absB{\int_{{Q}} f(y)\varphi(y)\dd y} \ind_Q(x) \leq \MaxAvg^{s_0}f(x).
$$
The middle expression can be extended from $f \in L^1_{\loc}(\R^d)$ to (tempered) distributions.
    \end{enumerate}
\end{remark}

To prove Theorem \ref{th:theoremC} using $f_Q$ as in \eqref{eq:choicefQ}, we need to estimate both
\begin{align*}
    \per_Q^{r}\brb{\abs{T(f\psi_{2Q})}} \qquad \text{ and } \qquad \per_Q^r(\MaxAvg^{\#}_Qf)
\end{align*}
in terms of $\per_Q^{r}(\MaxAvg^s_{3Q}f)$. For the former term, which is the local contribution of $T$, this is contained in the following proposition.

\begin{proposition}\label{prop:estonTF}
    Let $r \in (0,\frac12]$, $s>0$ and let $T$ be an $s$-smooth Calder\'on--Zygmund operator. For any $Q \subseteq \R^d$ and all $f \in L^\infty_c(\R^d)$ supported on $3Q$  we have
    \begin{equation*}
\per_Q^{2r}\br{\abs{Tf}} \lesssim_{T,s} \tfrac{1}{r} \cdot\per_Q^{r}(\MaxAvg^s_{3Q}f)
    \end{equation*}
\end{proposition}

\begin{proof}
Fix $\psi \in C^\infty_c(\R^d)$ such that $0\leq \psi \leq 1$,  $\psi\equiv 1$ on $[-\frac12,\frac12]^d$ and $\psi\equiv 0$ outside $(1+\frac{1}{8\sqrt{d}})[-\frac12,\frac12]^d$. For each cube $Q\subseteq \R^d$ define 
$$
\psi_Q(x):= \psi\brB{\frac{x-x_Q}{\ell(Q)}}, \qquad x \in \R^d.
$$
Note that $\psi_Q\equiv 1$ on $Q$.

Fix a cube $Q \subseteq \R^d$ and $f \in L^\infty_c(\R^d)$ supported on $3Q$. Abbreviate $m_Q:=\per_Q^{r}(\MaxAvg_{3Q}^s f)$ and for $\lambda>0$ define
\begin{align*}
E_\lambda&:=\cbraceb{x \in Q: \abs{Tf(x)}>\lambda m_Q \text{ and } {\MaxAvg_{3Q}^s f(x)}\leq m_Q}.
\end{align*}
If there exists a $\lambda>0$, depending only on $T,r,s,d$,  such that
\begin{align}\label{eq:Testimatelambda}
    \abs{E_\lambda}\leq r\abs{Q},
\end{align}
we have 
    \begin{align*}
         \absb{\cbraceb{x \in Q:  \abs{Tf(x)}>\lambda m_Q}} &\leq \abs{E_\lambda}+ \absb{\cbraceb{x \in Q: \MaxAvg_{3Q}^s f(x) >m_Q}}    
         \leq 2 r\abs{Q}.
    \end{align*}
and therefore $\per_Q^{2r}\br{\abs{Tf}} \leq \lambda m_Q$, finishing the proof.

It remains to show \eqref{eq:Testimatelambda}.  Define the open set
\begin{align*}
    \Omega := \cbraceb{x \in \R^d: \MaxAvg_{3Q}^s f(x) >m_Q}
\end{align*}
and note that $\Omega\subsetneq 3Q$. Let $\cbrace{R_j}_j$ be a Whitney decomposition of $\Omega$ (cf. \cite[Section J.1]{Gra14}), that is,
\begin{align*}
\textstyle\bigcup_jR_j &=\Omega , \\
\sqrt{d}\cdot \ell(R_j)\leq \dist(R_j, 3Q\setminus \Omega) &\leq 4\sqrt{d}\cdot \ell(R_j),\\
\textstyle\sum_j \ind_{\frac{9}{8}R_j} &\leq 12^d.
\end{align*}
In particular, we have $\supp(\psi_{R_j}) \subseteq \Omega$. Define
\begin{align*}
\eta_j(x) := \frac{\psi_{R_j}(x)}{\sum_k \psi_{R_k}(x)} \ind_\Omega(x), \qquad x \in \R^d,
\end{align*}
so that we have
\begin{align*}
    \textstyle\sum_j\eta_j &= \ind_\Omega,\\
    \int_{\R^d}{\eta_j(x)}\dd x&\eqsim \abs{R_j} ,\\
    \abs{\partial^\beta\eta_j} &\lesssim_{\beta} \ell(R_j)^{-\abs{\beta}}, \qquad \beta\in \N^d.
\end{align*}
Now define the smooth Calder\'on--Zygmund decomposition of $f$ as
\begin{align*}
c_j&:=\frac{1}{\int_{\R^d}\eta_j(y)\dd y}\int_{\R^d} f(y) \eta_j(y)\dd y,\\
b_j(x)&:= f(x)\eta_j(x) - c_j \eta_j(x), &&x \in \R^d,\\
g(x)&:= f(x)-\sum_j b_j(x), &&x \in \R^d.
\end{align*}
and note that the support of $g$ is contained in $3Q$. 
Then, using any $z \in {3Q}\setminus\Omega$ such that $\dist(z,R_j)\lesssim \ell(R_j)$, we have 
\begin{align}\label{eq:badpartest}
    \abs{c_j} \lesssim \frac{1}{\abs{R_j}}\absB{\int_{\R^d} f(y) \eta_j(y)\dd y}\lesssim_d  \MaxAvg_{3Q}^s f(z) \leq m_Q.
\end{align}
Moreover, by a smooth variant of the Lebesgue differentiation theorem, for a.e. $x \in {3Q} \setminus \Omega$ we have 
$$
\abs{f(x)}\lesssim \MaxAvg_{3Q}^s f(x) \leq m_Q
$$
and therefore we have for all $x \in 3Q$ that
$$\abs{g(x)} \leq \abs{f(x)} \ind_{{3Q}\setminus \Omega}(x) + \sum_j \abs{c_j}\eta_j(x)\lesssim m_Q.$$
Combined with the $L^2$-boundedness of $T$, this yields
\begin{align}\label{eq:estforg}
\begin{aligned}
    \abs{E_\lambda}&\lesssim  \frac{1}{\lambda^2 m_Q^2} \int_Q \abs{Tg(x)}^2\dd x &&+ \sum_j \frac{1}{\lambda m_Q} \int_{E_\lambda} \abs{Tb_j(x)}\dd x\\
&\leq  \frac{\nrm{T}_{L^2(\R^d) \to L^2(\R^d)}^2}{\lambda^2 m_Q^2} \int_{{3Q}} \abs{g(x)}^2\dd x &&+ \sum_j \frac{1}{\lambda m_Q} \int_{E_\lambda} \abs{Tb_j(x)}\dd x\\
&\lesssim   \frac{\nrm{T}_{L^2(\R^d) \to L^2(\R^d)}^2}{\lambda^2} \cdot\abs{Q} &&+ \sum_j \frac{1}{\lambda m_Q} \int_{E_\lambda} \abs{Tb_j(x)}\dd x.
\end{aligned}
\end{align}
It remains to estimate the sum in $j$. Fix a $j$ and denote the center of $R_j$ by  $x_j$. Furthermore, fix an $x \in E_\lambda\subseteq \Omega^c$ and note that $\dist\br{x,R_j}\geq \sqrt{d}\ell(R_j)$. Write $s=m+\delta$ with $m \in \N$ and $\delta \in (0,1]$ and define
\begin{align*}
\xi_j(y)&:= \frac{\abs{x-x_j}^{d+\delta}}{\ell(R_j)^\delta} \brb{K(x,y) - K\br{{x,x_j}}}, && y \in \R^d,\\
    \zeta_j(y)&:= \xi_j(y)\cdot \psi_{(1+\frac{1}{8\sqrt{d}})R_j}\br{y}, && y \in \R^d.
\end{align*}
For $y \in \supp(\zeta_j)$ we have
\begin{align*}
    \abs{y-x_j} &\leq \brb{1+\tfrac{1}{8\sqrt{d}}}^2 \cdot \tfrac{\sqrt{d}}2  \cdot \ell(R_j) \leq \brb{1+\tfrac{1}{2\sqrt{d}}} \cdot \tfrac{\sqrt{d}}2 \cdot \ell(R_j)\\
    \abs{x-x_j}&\geq (\tfrac12 +\sqrt{d}) \cdot \ell(R_j) = \brb{1+\tfrac{1}{2\sqrt{d}}} \cdot {\sqrt{d}} \cdot \ell(R_j),
\end{align*}
i.e., $\abs{y-x_j} \leq \frac12 \abs{x-x_j}$.  Therefore,  the assumptions on $K$ yield for all $y \in \supp(\zeta_j)$
\begin{align*}
    \abs{\xi_j(y)} &\lesssim_K  \frac{\abs{x-x_j}^{d+\delta}}{\ell(R_j)^\delta} \frac{\abs{y-x_j}^\delta}{\abs{x-x_j}^{d+\delta}} \lesssim 1,\intertext{and similarly, for $1\leq \abs{\alpha}\leq m$}
    \abs{\partial^\alpha \xi_j(y)} &\lesssim_K  \frac{\abs{x-x_j}^{d+\delta}}{\ell(R_j)^\delta} \frac{1}{\abs{x-y}^{d+\abs{\alpha}}} \lesssim \frac{1}{\ell(R_j)^\delta} \frac{1}{\abs{x-x_j}^{\abs{\alpha}-\delta}}\lesssim \frac{1}{\ell(R_j)^{\abs{\alpha}}}.
    \intertext{
    Moreover for all $y,y' \in \supp(\zeta_j)$ and $\abs{\alpha}= m$ we have}
    \abs{\partial^\alpha\xi_j(y)-\partial^\alpha\xi_j(y')} &\lesssim_K \frac{\abs{x-x_j}^{d+\delta}}{\ell(R_j)^\delta} \frac{\abs{y-y'}^\delta}{\abs{x-y}^{d+\abs{\alpha}+\delta}} \lesssim \frac{1}{\ell(R_j)^\delta} \frac{\abs{y-y'}^\delta}{\abs{x-x_j}^{\abs{\alpha}}}\lesssim \frac{\abs{y-y'}^{\delta}}{\ell(R_j)^{\abs{\alpha}+\delta}}.
\end{align*}
By the (fractional) product rule, this proves that $\zeta_j$ is, up to a multiplicative constant depending on $K$ and $d$, an element of $\mathcal{F}_s(2R_j)$. Since $\int_{\R^d}b_j(x)\dd x = 0$, we can estimate 
\begin{align*}
    \abs{Tb_j(x)} &= \absB{\int_{\R^d} K(x,y)b_j(y)\dd y}\\
    &=\absB{\int_{\R^d} \brb{ K(x,y) - K(x,x_j)}b_j(y)\dd y}\\
    &= \frac{\ell(R_j)^\delta}{\abs{x-x_j}^{d+\delta}} \cdot \absB{\int_{\R^d}b_j(y)\zeta_j(y)\dd y}\\
    &\lesssim \frac{\ell(R_j)^\delta}{\abs{x-x_j}^{d+\delta}} \cdot \absB{\int_{\R^d}f(y)\eta_j(y)\zeta_j(y)\dd y}+ \frac{\ell(R_j)^\delta}{\abs{x-x_j}^{d+\delta}} \abs{R_j} \abs{c_j}\\
    &\lesssim_{K}  \abs{R_j} \cdot\frac{\ell(R_j)^\delta}{\abs{x-x_j}^{d+\delta}}\cdot m_Q
\end{align*}
where, in the final estimate,  we used \eqref{eq:badpartest} and a similar estimate for the first term.
We conclude that
\begin{align*}
\sum_j \frac{1}{\lambda m_Q} \int_{E_\lambda} \abs{Tb_j(x)}\dd x \lesssim_{K} \sum_j \frac{\abs{R_j}}{\lambda} \int_{R_j^c} \frac{\ell(R_j)^\delta}{\abs{x-x_j}^{d+\delta}} \dd x \lesssim_{s} \sum_j \frac{\abs{R_j}}{\lambda}\lesssim \frac{\abs{Q}}{\lambda}.
\end{align*}
Combined with \eqref{eq:estforg}, this yields  $
\abs{E_\lambda}\lesssim_{T,s} (\lambda^{-2}+\lambda^{-1})\abs{Q}.
$
Taking $\lambda>0$ large enough we obtain \eqref{eq:Testimatelambda}, which finishes the proof.
\end{proof}

We are now ready to prove Theorem \ref{th:theoremC}.
\begin{proof}[Proof of Theorem \ref{th:theoremC}]
Let $f \in L^\infty_c(\R^d)$ be supported on a cube $Q_0 \subseteq \R^d$, and let $r:= 2^{-d-4}$. For each cube $Q\subseteq \R^d$ let $\psi_Q$ be as in Proposition \ref{prop:estonTF} and set
$$
f_Q := T(f\psi_{2Q}), \qquad f_{Q,R}=f_Q-f_R= T(f(\psi_{2Q}-\psi_{2R})).
$$
Noting that $f_{Q_0} = Tf$ on $Q_0$ and combining Theorem \ref{th:theorem32Emiel} and Proposition \ref{prop:estonTF}, we find a sparse family ${\Ss}\subseteq \D(Q_0)$ such that
$$
|Tf(x)|\lesssim_{T,s}  \sum_{Q\in {\Ss}}\left[\per_Q^{r}(\MaxAvg^s_{3Q}f)+\per_Q^{2r}(\MaxAvg^{\#}_{T,Q}f)\right]\ind_Q(x), \qquad x \in Q_0,
$$
where $\MaxAvg^{\#}_{T,Q}$ is as in \eqref{eq:grandmaxtrunc} with our specific choice for $f_Q$.
It remains to show that for all cubes $Q \in \D(Q_0)$ 
$$
\per_Q^{2r}(\MaxAvg^{\#}_{T,Q}f) \lesssim_{T,s} \per_Q^{r}(\MaxAvg^s_{3Q}f),
$$
for which it suffices to show that 
\begin{equation}\label{eq:themCtoshow}
    \MaxAvg^{\#}_{T,Q}f(x) \lesssim_{K,s} \MaxAvg_{3Q}^sf(x), \qquad x \in Q.
\end{equation}

 Fix a cube $Q \in \D(Q_0)$ and $x \in Q$. Take $R \in \D(Q)$ such that $x \in R$ and let $x',x'' \in R$. Let $n\in \N$ be the largest integer such that $2^nR\subseteq 2Q$ and note that $2^n\ell(R) \eqsim \ell(Q)$. Write $s = m+\delta$ for $m \in \N$ and $\delta \in (0,1]$ and for $k=1,\ldots,{n}$ define
\begin{align*}
\xi_k(y)&:= 2^{k(d+\varepsilon)}\abs{R} \cdot\brb{ K(x',y) - K(x'',y)}, \qquad &&y \in \R^d,\\
   \zeta_k(y)&:= \xi_k(y)\cdot \brb{\psi_{2^{k+1}R}(y) - \psi_{2^{k} R}(y)}, \qquad &&y \in \R^d,\, 1\leq k\leq n-1\\
   \zeta_n(y)&:= \xi_n(y)\cdot \brb{\psi_{2Q}(y) - \psi_{2^{n} R}(y)}, \qquad &&y \in \R^d.
\end{align*}
For $y \in \supp(\zeta_k)$ we note that $$\abs{x-y}\eqsim\abs{x'-y}\eqsim \abs{x''-y} \eqsim 2^k \ell(R),$$ so the assumptions on $K$ yield
 for all $y,y' \in \supp(\zeta_k)$ and $0\leq \abs{\alpha} \leq m$
\begin{align*}
    \abs{\partial^\alpha \xi_k(y)} &\lesssim_K  2^{k(d+\varepsilon)}\abs{R} \cdot \frac{\abs{x'-x''}^\varepsilon}{\abs{x'-y}^{d+\abs{\alpha}+\varepsilon}}
    \lesssim \frac{1}{(2^k\ell(R))^{\abs{\alpha}}},\intertext{and for $\abs{\alpha}=m$}
    \abs{\partial^\alpha\xi_k(y)-\partial^\alpha\xi_k(y')} &\lesssim_K 2^{k(d+\varepsilon)}\abs{R} \cdot \frac{\abs{y-y'}^\delta\abs{x'-x''}^\varepsilon}{\abs{x'-y}^{d+\abs{\alpha}+\delta+\varepsilon}} \lesssim  \frac{\abs{y-y'}^\delta}{(2^k\ell(R))^{\abs{\alpha}+\delta}}.
\end{align*}

By the (fractional) product rule, this proves that $\zeta_k$ is, up to a multiplicative constant depending on $K$ and $d$, an element of $$\mathcal{F}_s\brb{(1+\tfrac{1}{8\sqrt{d}}) 2^{k+1}R}$$ for $k=1,\ldots,n-1$. Furthermore, $\zeta_n$ is, up to a multiplicative constant depending on $K$ and $d$, an element of $\mc{F}_s(3Q)$. Therefore
\begin{align*}
  \abs{  f_{Q,R}(x')-f_{Q,R}(x'')}&= \absB{\int_{\R^d} \brb{ K(x',y) - K(x'',y)}(\psi_{2Q}(y) - \psi_{2R}(y))f(y)\dd y}\\
    &\leq \sum_{k=1}^n 2^{-k\varepsilon}\frac{1}{\abs{2^kR}}\absB{\int_{\R^d} \zeta_k(y)f(y)\dd y}\\
    &\lesssim_{K,\varepsilon} \MaxAvg_{3Q}^sf(x).
\end{align*}
Taking the supremum over all $x',x'' \in R$ and $R \in \D(Q)$, we obtain \eqref{eq:themCtoshow}. This finishes the proof.
\end{proof}

\section{Weighted norm estimates} \label{sec:sec4}

In this section we show how our cancellative sparse bounds lead to quantitative, and in some cases sharp, weighted Hardy space estimates. We begin with the $s$-smooth Calderón--Zygmund operators studied in Section \ref{sec:sec3}, as this case contains the main ideas and technical difficulties. In Remark \ref{rem:nonCZweight} we indicate how the arguments can be adapted to the martingale setting of Sections \ref{sec:sec1}-\ref{sec:sec2}.

\subsection{Muckenhoupt weighted Hardy spaces}
By a weight $w$ we mean a locally integrable $w \colon \R^d \to (0,\infty)$. For a measurable set $E\subseteq \R^d$ we write $w(E):= \int_E w(x)\dd x$. For $p \in (0,\infty)$ we define $L^p(w)$ as the space of all measurable $f \colon \R^d \to \C$ such that
$$
\nrm{f}_{L^p(w)}:= \brB{\int_{\R^d} \abs{f(x)}^p w(x)\dd x}^{1/p}<\infty.
$$
For $q \in (1,\infty)$ we say that $w$ belongs to the Muckenhoupt $A_q$-class and write $w \in A_q$ if
$$
[w]_{A_q}:=\sup_{Q} \,\ip{w}_{Q} \ip{w^{-{\frac{1}{q-1}}}}_{Q}^{q-1}<\infty,
$$
where the supremum is taken over all cubes $Q\subseteq \R^d$.
Furthermore, we say that $w$ belongs to the Muckenhoupt $A_\infty$-class and write $w \in A_{\infty}$ if
$$
[w]_{A_{\infty}}:=\sup_{Q}\frac{1}{w(Q)}\int_Q \sup_{R\subseteq Q} \ip{w}_R \ind_R(x) \dd x.$$
By \cite[Proposition 2.2]{HP13} we know 
  $[w]_{A_{\infty}}\lesssim [w]_{A_q},$
and thus $A_q\subseteq A_\infty$. In fact, we have $\bigcup_{q> 1} A_q = A_\infty$, i.e. for all $w \in A_\infty$ there exists a $q \in (1,\infty)$ such that $w \in A_q$. 
However,  both $q$ and $[w]_{A_q}$ can  grow exponentially with respect to $[w]_{A_\infty}$, see \cite{HP16} and the references therein. Therefore, to obtain quantitative weighted norm inequalities, we will assume $w \in A_q$ for some $q \in (1,\infty)$. The key property of $w \in A_q$ we will use is that, for a cube $Q \subseteq \R^d$ and $E\subseteq Q$, we have 
\begin{equation}\label{eq:RDq}
\brB{\frac{\abs{E}}{\abs{Q}}}^q \leq [w]_{A_q}  \cdot \frac{w(E)}{w(Q)}.
\end{equation}

Next, we introduce the weighted Hardy spaces $H^p(w)$ following \cite{ST89}, see in particular Chapters VI and VIII therein for further details. Let $\mc{S}(\R^d)$ and $\mc{S}'(\R^d)$ denote the spaces of Schwartz functions and  tempered distributions, respectively.  Fix a $\psi \in \mc{S}(\R^d)$ with $\int_{\R^d}\psi(x)\dd x =1$. For $t>0$ set $\psi_t(x):= t^{-d}\psi(x/t)$ and for $f \in  \mc{S}'(\R^d)$ define
\[
\MaxAvg^\psi f(x) := \sup_{\substack{ y \in \R^d,\, t>0:\\ \abs{x-y}\leq t} }\abs{\psi_t*f(y)}, \qquad x \in \R^d.
\]
For $p \in (0,\infty)$ and $w \in A_\infty$, we define
\[
H^p(w):= \cbraceb{f \in \mc{S}'(\R^d) : \MaxAvg^\psi f \in L^p(w)}
\]
with quasi-norm
\[
\nrm{f}_{H^p(w)}:= \nrm{ \MaxAvg^\psi f}_{L^p(w)}.
\]
Then $H^p(w)$ is a quasi-Banach space and different choices of $\psi$ yield equivalent quasi-norms. We refer to \cite[Chapters VI and VIII]{ST89} for numerous other equivalent quasi-norms on $H^p(w)$. Since the constants in these equivalences depend on the weight characteristic of 
$w$, one must fix a specific quasi-norm in order to obtain \emph{quantitative} weighted Hardy space estimates. We shall therefore work with the quasi-norm defined via the smooth maximal operator $\MaxAvg^s$ introduced in Section \ref{sec:sec3}. As we show in the next proposition, this quasi-norm is equivalent to the one defined above.

\begin{proposition}\label{prop:Mschar}
Let $p \in (0,\infty)$, $s>0$ and assume $q :=p(1+\frac{s}d)>1$. For $w \in A_q$ and  $f \in 
H^p(w) \cap L^1_{\loc}(\R^d)$ we have 
    $$
\nrm{f}_{\Hardy^p(w)} \lesssim_{s,\psi} \nrm{\MaxAvg^s f}_{L^p(w)} \lesssim_{p,s,w,\psi} \nrm{f}_{\Hardy^p(w)}
    $$
\end{proposition}

Before turning to the proof, we record a few remarks concerning Proposition~\ref{prop:Mschar}.
\begin{remark}~\label{rem:Mschar}
\begin{enumerate}[(i)]
    \item We restrict to the dense subspace $H^p(w) \cap L^1_{\loc}(\R^d)$ in Proposition~\ref{prop:Mschar}, since $\MaxAvg^s f$ is not necessarily well-defined for arbitrary $f \in \mc{S}'(\R^d)$. Alternatively, one can modify the definition of $\MaxAvg^s$ by restricting the supremum to test functions in $\mc{F}_s(Q)\cap C^\infty_c(\R^d)$, see also Remark~\ref{rem:defMs}\ref{it:defMsiii}.
    \item We do not make the dependence of the implicit constant on $w$ explicit in the second estimate of Proposition~\ref{prop:Mschar}, although the proof shows that it depends only on $[w]_{A_q}$. More precisely, inspection of the argument yields
    \[
\nrm{\MaxAvg^s f}_{L^p(w)} \lesssim_{p,s,\psi} [w]_{A_q}^{\max\cbrace{1,\frac{q'}{p}}} \cdot\nrm{f}_{H^p_{\mathrm{at}}(w)},
    \]
    where $\nrm{\cdot}_{H^p_{\mathrm{at}}(w)}$ denotes the equivalent quasi-norm on $H^p(w)$ induced by the atomic decomposition. As far as the authors are aware, the quantitative dependence on $[w]_{A_q}$ in the estimate
    \[
\nrm{\cdot}_{H^p_{\mathrm{at}}(w)} \lesssim_{p,w,\psi} \nrm{\cdot}_{{H^p}(w)}
    \]
    has not been tracked in the literature.
\end{enumerate}
\end{remark}

\begin{proof}[Proof of Proposition \ref{prop:Mschar}]
Take $f \in 
H^p(w) \cap L^1_{\loc}(\R^d)$. We first note that 
$$
\MaxAvg^{\psi}f(x) \lesssim_{s,\psi} \MaxAvg^sf(x), \qquad x \in \R^d,
$$
which is trivial if $\psi$ has compact support and otherwise follows by decomposing into annuli and using that $\psi \in \mc{S}(\R^d)$. This proves the first claimed norm estimate. 

For the second estimate, 
define $N:=\ceil{s}$ and note that by density (see \cite[Chapter~VII]{ST89}) we may assume without loss of generality that $f \in C_c^\infty(\R^d)$ and 
$$
\int_{\R^d} x^\alpha f(x) \dd x =0, \qquad \abs{\alpha} \leq N.
$$
Now, by the finite atomic decomposition of $\Hardy^p(w)$ (see, e.g., \cite[Theorem 2.6]{CMN19}), we can write $f = \sum_{k=1}^n \lambda_k a_k$ for a sequence of non-negative numbers $\cbrace{\lambda_k}_{k=1}^n$ and a sequence of $(N,\infty)$-atoms $\cbrace{a_k}_{k=1}^n$ supported on cubes $\cbrace{Q_k}_{k=1}^n$ with
\begin{equation}\label{eq:atomicest}
    \nrmB{\sum_{k=1}^n \lambda_k\ind_{Q_k}}_{L^p(w)} \lesssim_{p,w,\psi} \nrm{f}_{H^p(w)}.
\end{equation}
Here an $(N,\infty)$-atom (normalized as in \cite{CMN19}) is a function $a$ supported on a cube $Q$ such that $\nrm{a}_{L^\infty(Q)} \leq 1$ and
$$
\int_{\R^d} x^\alpha a(x) \dd x =0, \qquad  \abs{\alpha} \leq N.
$$

Fix $1\leq k \leq n$ and denote the center of $Q_k$ by $x_k$. We trivially have 
\begin{align*}
     \MaxAvg^s a_k(x)\cdot  \ind_{3Q_k}(x) \leq \nrm{a_k}_{L^\infty(Q_k)} \cdot  \ind_{3Q_k}(x) \lesssim_{s} \nrm{a_k}_{L^\infty(Q_k)} \cdot  M (\ind_{Q_k}) (x)^{\frac{d+s}{d}},
\end{align*}
where $M$ denotes the Hardy--Littlewood maximal operator.
For $x \notin 3Q_k$, let $R\subseteq \R^d$ be a cube such that $x \in R$ and $R\cap Q_k\neq \emptyset$.  Then $\ell(Q_k) \leq \ell(R)$ and
$$
\abs{x-x_k} \leq \sqrt{d}\brb{\ell(Q_k)+\ell(R)} \lesssim \ell(R).
$$
Write $s = m+\delta$ with $m \in \N$ and $\delta \in (0,1]$. Take $z \in R\cap Q_k$ and for $\varphi \in \mc{F}_s(R)$ let $P_{z}^m (\varphi)$ denote the Taylor polynomial of $\varphi$ at $z$ of degree $m$. Since $\varphi$ is $s$-smooth and supported on $R$, we have
$$
\abs{\varphi(y) - P_{z}^m (\varphi)(y)} \lesssim_{s} \frac{\abs{y-z}^s}{\ell(R)^s}, \qquad y \in \R^d,
$$
and therefore, using the cancellative properties of $a_k$, we have
\begin{align*}
\sup_{\varphi \in \mathcal{F}_s(R)} \,\frac{1}{\abs{R}} \absB{\int a_k(y)\varphi(y)\dd y} &= \sup_{\varphi \in \mathcal{F}_s(R)} \,\frac{1}{\abs{R}} \absB{\int_{Q_k} a_k(y)\brb{\varphi(y)-P_{z}^m (\varphi)(y)}\dd y}\\
&\lesssim_{s} \frac{\nrm{a_k}_{L^1(Q_k)}}{\abs{R}} \frac{{\ell(Q_k)}^s}{\ell(R)^s} \lesssim  \nrm{a_k}_{L^\infty(Q_k)} \cdot \frac{\ell(Q_k)^{d+s}}{\abs{x-x_k}^{d+s}}\\&\eqsim \nrm{a_k}_{L^\infty(Q_k)} \cdot M(\ind_{Q_k})(x)^{\frac{d+s}{d}}.
\end{align*} 
Hence, denoting $t:=\frac{d}{d+s}$, we have
$$
\MaxAvg^s f(x) \leq  \sum_{k=1}^n \lambda_k \cdot \MaxAvg^sa_k(x) \lesssim_{s}   \sum_{k=1}^n \lambda_k \cdot M(\ind_{Q_k})^{1/t} =  \sum_{k=1}^n  M(\lambda_k^{t}\ind_{Q_k})^{1/t}.
$$
By assumption, we have $p/t=q> 1$ and  $1/t> 1$, so by the weighted Fefferman-Stein inequality \cite[Theorem 1.12]{CMP12} we obtain
\begin{align*}
    \nrm{\MaxAvg^s f}_{L^p(w)} &\lesssim_{s}  \nrmB{\sum_{k=1}^n  M(\lambda_k^{t}\ind_{Q_k})^{1/t}}_{L^p(w)}\\
    &=\nrmB{\brB{\sum_{k=1}^n  M(\lambda_k^{t}\ind_{Q_k})^{1/t}}^{t}}_{L^{p/t}(w)}^{1/t}\\
    &\lesssim_{p,s} [w]_{A_{p/t}}^{\frac1t \cdot \max\cbrace{t,\frac{t}{p-t}}} \nrmB{\brB{\sum_{k=1}^n  \lambda_k\ind_{Q_k}}^{t}}_{L^{p/t}(w)}^{1/t}\\
    &= [w]_{A_{q}}^{\max\cbrace{1,\frac{q'}{p}}} \nrmB{\sum_{k=1}^n  \lambda_k\ind_{Q_k}}_{L^{p}(w)}.
\end{align*}
Combined with \eqref{eq:atomicest}, this finishes the proof.
\end{proof}

\subsection{Weighted norm estimates for Calder\'on--Zygmund operators}
We now turn to weighted norm estimates for operators from $H^p(w)$ to $L^p(w)$ for $0<p<\infty$ and $w \in A_\infty$. Combining Theorem \ref{th:theoremC} and Proposition \ref{prop:Mschar}, this has been reduced to the $L^p(w)$-boundedness of the cancellative sparse operator given by 
$$
f \mapsto \sum_{Q\in \mc{S}} \per^r_Q(|f|)\ind_Q.
$$
The following proposition contains the key estimate.
\begin{proposition}\label{prop:weightedsparse}
    Let $\mc{S}$ be an $\eta$-sparse collection of cubes. Let $p,t \in (0,\infty)$, $q \in (1,\infty)$, $r \in (0,1)$ and $w \in A_q$. Then for all $f \in L^p(w)$
    \begin{equation*}
        \nrmB{\brB{ \sum_{Q\in \mc{S}} \per^r_Q(|f|)^t\ind_Q}^{1/t} }_{L^p(w)} \lesssim_{p,q,r,t,\eta} [w]_{A_q}^{\frac1p} [w]_{A_\infty}^{(\frac1{t}-\frac1{p})_+}\cdot \nrm{f}_{L^p(w)}.
    \end{equation*}
\end{proposition}

\begin{proof}
We first note that it suffices to prove the statement for $p=t$. Indeed, the case $p<t$  follows from the embedding $\ell^p \hookrightarrow \ell^t$ and the case $p>t$ from \cite[Theorem 4.2]{NSS24} and the consequence stated directly below. Furthermore, using \cite[Lemma 6.6]{LN15}, we may assume without loss of generality that $\mc{S}$ is $(1-\frac{r}{2})$-sparse. Finally, we may assume that $\per^r_Q(|f|)>0$ for all $Q \in \mathcal{S}$.

Take $p=t$ and for every $Q \in \mc{S}$ let $E_Q\subseteq Q$ be such that $\abs{E_Q}\geq (1-\frac{r}{2})\abs{Q}$ and such that the $E_Q$'s are disjoint. Furthermore, let $$F_Q:= \cbraceb{x \in Q:|f(x)|>\tfrac12\per^r_Q(|f|)},$$ for which we note that $\abs{F_Q} \geq r\abs{Q}$ by definition. Hence, for $G_Q:=E_Q\cap F_Q$ we have
\[
\abs{G_Q} = \abs{F_Q \setminus E_Q^{\mathrm{c}}}\geq r\abs{Q} - \tfrac{r}{2}\abs{Q}  = \tfrac{r}{2}\abs{Q}. 
\]
Since $w \in A_q$, we deduce from \eqref{eq:RDq} that
\[
w(G_Q) \geq [w]_{A_q}^{-1} \brb{\tfrac{r}{2}}^q \cdot w(Q).
\]
We can now calculate
\begin{align*}
     \nrmB{\brB{\sum_{Q\in \mc{S}} \per^r_Q(|f|)^p \ind_Q}^{1/p} }_{L^p(w)}^p &= \sum_{Q \in \mc{S}} \per^r_Q(|f|)^p w(Q)\\
     &\lesssim_{r,q} [w]_{A_q} \sum_{Q \in \mc{S}} \per^r_Q(|f|)^p w(G_Q)\\
     &\lesssim_p [w]_{A_q} \sum_{Q \in \mc{S}} \int_{G_Q} |f(x)|^pw(x)\dd x\\
     &\leq [w]_{A_q} \int_{\R^d} |f(x)|^pw(x)\dd x = [w]_{A_q}\nrm{f}_{L^p(w)}^p,
\end{align*}
finishing the proof.
\end{proof}

Before using Proposition \ref{prop:weightedsparse} to deduce $H^p(w)$ to $L^p(w)$ estimates for Calder\'on--Zygmund operators, we note that the weight dependence in Proposition \ref{prop:weightedsparse} is sharp. Recall that $[w]_{A_\infty}\lesssim_q [w]_{A_q}$ for all $q \in [1,\infty)$.
\begin{lemma}
     Let $p,t \in (0,\infty)$, $q \in (1,\infty)$, $r \in (0,1)$ and $w \in A_q$. If $\alpha \geq 0$ is such that for all $\eta$-sparse collections of cubes $\mc{S}$ and all $f \in L^p(w)$ we have
    \begin{equation*}
        \nrmB{\brB{\sum_{Q\in \mc{S}} \per^r_Q(|f|)^t \ind_Q}^{1/t} }_{L^p(w)} \lesssim [w]_{A_q}^{\alpha}\cdot \nrm{f}_{L^p(w)},
    \end{equation*}
     with implicit constant independent of $w$, then $\alpha \geq \max\cbrace{\frac1p,\frac1t}$.
\end{lemma}
\begin{proof} By 
\cite[Lemma 6.6]{LN15}, we may assume without loss of generality that $\eta=\frac12$ and by monotonicity we may assume $r<\frac12$.
   For all $k \geq 0$ define $Q_k:= [0,2^{-k})^d$ and set $\mc{S}:=\cbrace{Q_k:k\geq 0}$, which is $\frac12$-sparse.

   First let $R\subseteq Q_0$ such that $\abs{R} = 2r\abs{Q_0}$. For $\varepsilon\in(0,1)$ define
   \begin{align*}
       w_\varepsilon:= \varepsilon\ind_R +\ind_{\R^d\setminus R},
   \end{align*}
   for which we have $[w_\varepsilon]_{A_q} \eqsim \varepsilon^{-1}$. Take $f_{\varepsilon} = \varepsilon^{-1/p} \ind_R$, for which we note that $$\nrm{f_\varepsilon}_{L^p(w_\varepsilon)} =\abs{R}^\frac1p\eqsim_{p,r} 1.$$ Finally, we calculate
   $$
\nrmB{\brB{\sum_{Q\in \mc{S}} \per^r_Q(|f_\varepsilon|)^t \ind_Q}^{1/t} }_{L^p(w_\varepsilon)}  \geq \nrmb{\per^r_{Q_0}(|f_{\varepsilon}|) \ind_{Q_0}}_{L^p(w_\varepsilon)}  = \varepsilon^{-1/p}\cdot w_\varepsilon(Q_0)^{1/p} \eqsim_{p} [w_\varepsilon]_{A_q}^{1/p},
   $$
   so $\alpha \geq \frac1p$.

   Next, for $\varepsilon\in (0,1)$ define
   $$
w_{\varepsilon}(x):= \varepsilon \abs{x}^{-d+\varepsilon} \ind_{Q_0}(x) + \ind_{\R^d\setminus Q_0}(x), \qquad x \in \R^d,
   $$ for which we have $[w_\varepsilon]_{A_q} \eqsim \varepsilon^{-1}$. Furthermore, define $f:=\ind_{Q_0}$, which satisfies $\nrm{f}_{L^p(w_{\varepsilon})}\eqsim 1$. We again calculate
   \begin{align*}
       \nrmB{\brB{\sum_{Q\in \mc{S}} \per^r_Q(|f|)^t \ind_Q}^{1/t} }_{L^p(w_{\varepsilon})}&\geq \nrmB{\brB{\sum_{k=0}^\infty\ind_{Q_k}}^{1/t} }_{L^p(w_{\varepsilon})} \\&\geq \brB{\sum_{k=1}^\infty k^{p/t} w(Q_{k-1} \setminus Q_{k})}^{1/p}\\
       &\eqsim_{p} \brB{\sum_{k=1}^\infty k^{p/t} \varepsilon 2^{-k\varepsilon}}^{1/p}\eqsim_{p,t} \varepsilon^{-1/t} \eqsim [w_{\varepsilon}]_{A_q}^{1/t},
   \end{align*}
   so $\alpha \geq \frac1t$.
\end{proof}

Combining Theorem \ref{th:theoremC} with Propositions \ref{prop:Mschar} and \ref{prop:weightedsparse}, we get quantitative weighted estimates for $s$-smooth Calder\'on--Zygmund operators from $H^p(w)$ to $L^p(w)$.

\begin{theorem}\label{th:theoremCZO}
Let $s>0$ and let $T$ be an $s$-smooth Calder\'on-Zygmund operator. Let $p \in (0,\infty)$ and $w \in A_\infty$. For all $f \in L^\infty_c(\R^d)$  we have
\begin{align} \label{eq:TfweightA_infty}
\nrm{Tf}_{L^p(w)}  &\lesssim_{T,p,s}  [w]_{A_\infty} \cdot \nrmb{\sup_{Q} \per^r_Q(\MaxAvg^s_{3Q}f)\ind_Q}_{L^p(w)}\intertext{with $r\in (0,1)$ only depending on $d$. For any $q \in (1,\infty)$ such that $w \in A_q$, we have}
\nrm{Tf}_{L^p(w)}  &\lesssim_{T,p,q,s}  [w]_{A_q}^{\frac1p} [w]_{A_\infty}^{(1-\frac1{p})_+} \cdot \nrm{\MaxAvg^sf}_{L^p(w)}\label{eq:TfweightA_q}
\end{align}
In particular, if $s\geq d(\frac{q}{p}-1)$, then $T$ is bounded from $H^p(w)$ to $L^p(w)$.
\end{theorem}
\begin{proof}
\eqref{eq:TfweightA_infty} follows  by first  combining Theorem \ref{th:theoremC} with \cite[Theorem 4.2]{NSS24} and then expanding the supremum to all cubes $Q\subseteq \R^d$. Similarly, \eqref{eq:TfweightA_q} is a corollary of Theorem \ref{th:theoremC} combined with Proposition \ref{prop:weightedsparse} using $t=1$. The last statement follows from \eqref{eq:TfweightA_q} by Proposition~\ref{prop:Mschar} and the density of $H^p(w) \cap L^\infty_{c}(\R^d)$ in $H^p(w)$, which follows from the atomic decomposition of $H^p(w)$ (see \cite[Chapter 8]{ST89}).
\end{proof}

Qualitatively, the boundedness of $T$ from $H^p(w)$ to $L^p(w)$ in Theorem \ref{th:theoremCZO} is already known; see, for example, \cite{HO17} and the references therein.
The dependence on the weight characteristic in \eqref{eq:TfweightA_q} is sharp for $p \geq 1$, as we will show in Lemma \ref{lem:SharpHilber} below. Hence, Theorem \ref{th:theoremCZO} may be regarded as an $H^1(w)\to L^1(w)$ analogue of the $A_2$ theorem \cite{Hy12}. To the best of the authors' knowledge, Theorem \ref{th:theoremCZO} is the first quantitatively sharp weighted estimate for $s$-smooth Calder\'on--Zygmund operators in the Hardy space setting. 
Moreover, it recovers several results from the literature within a  cancellative framework.
For example:
\begin{itemize}
\item In \cite{CF74} it was shown  for $p \in (0,\infty)$ and $w \in A_\infty$ that
\begin{align*}
\nrm{Tf}_{L^p(w)} \lesssim_{T,p,w} \nrm{Mf}_{L^p(w)}, \qquad f \in L^\infty_c(\R^d).
\end{align*}
\eqref{eq:TfweightA_q} is an improvement of this result, replacing the Hardy--Littlewood maximal operator $M$ on the right-hand side by the cancellative maximal operator $\MaxAvg^s$, which is smaller. Furthermore, \eqref{eq:TfweightA_q}  provides (sharp) quantitative dependence on $[w]_{A_q}$.
\item By the weak $L^1$-boundedness of the Hardy--Littlewood maximal operator $M$, we can estimate $$\per^r_Q(\MaxAvg^s_{3Q}f)\ind_Q \leq \per^r_Q(M(f\ind_{3Q}))\ind_{3Q}   \lesssim \tfrac1r\, \ipb{\abs{f}}_{3Q} \ind_{3Q},$$  \eqref{eq:TfweightA_infty}  implies 
$$
\nrm{Tf}_{L^p(w)}  \lesssim_{T,p}  [w]_{A_\infty} \cdot \nrm{Mf}_{L^p(w)},\qquad f \in L^\infty_c(\R^d).
$$
which for $p=1$ was previously obtained in \cite[Lemma 2.1]{LOP2009}.
\end{itemize}

\begin{lemma}\label{lem:SharpHilber}
     Let $s>0$, $p \in (0,\infty)$, $q \in (1,\infty)$ and suppose $w \in A_q$. Suppose $s>\frac1p-1$ and let $H$ denote the Hilbert transform. If $\alpha \geq 0$ is such that for all $f \in H^p(w) \cap L^1_{\loc}(\R)$ we have
    \begin{equation*}
        \nrm{Hf}_{L^p(w)} \lesssim [w]_{A_q}^{\alpha}\cdot \nrm{\MaxAvg^sf}_{L^p(w)},
    \end{equation*}
    with implicit constant independent of $w$, then $\alpha \geq 1$.
\end{lemma}

\begin{proof}
    For $\varepsilon \in(0,1)$ define
    $$
w_\varepsilon(x):=\varepsilon x^{-1+\varepsilon} \ind_{(0,1)}(x)+\ind_{\R\setminus (0,1)}(x), \qquad x \in \R,
    $$
    for which we have
   $
[w_\varepsilon]_{A_q} \eqsim \varepsilon^{-1}.$ Write $s = m+\delta$ with $m\in \N$ and $\delta \in (0,1]$ and define
$$
f(x):= \sum_{k=0}^{m+1} (-1)^k {\binom{m+1}{k}}\ind_{(k,k+1)}(x), \qquad x\in \R,
$$
for which we note that
$$
\int_{\R} x^k f(x) \dd x =0, \qquad 0\leq k\leq m.
$$
Since $f$ has a jump discontinuity in $x=0$, we know that $\abs{Hf(x)} \eqsim \log (\frac1{|x|})$ for $x \in (-\frac12,\frac12)$. Therefore
\begin{align*}
    \nrm{Hf}_{L^p(w_\varepsilon)} &\gtrsim \brB{\int_0^\frac12 \log(\tfrac{1}{x})^p \varepsilon x^{-1+\varepsilon} \dd x}^{1/p}\gtrsim \varepsilon^{-1}\eqsim [w_\varepsilon]_{A_q}.
\end{align*}
Now to estimate $\MaxAvg^sf$, fix $x \in \R\setminus(-1,m+3)$ and take an interval $I$ containing $x$ such that $I\cap (0,m+2) \neq \emptyset$. Furthermore, fix a $\varphi \in \mc{F}_s(I)$ and let $P^m_0 (\varphi)$ denote the Taylor polynomial of $\varphi$ at $0$ of degree $m$. Since $\varphi$ is $s$-smooth and supported on $I$, we have
$$
\abs{\varphi(y) - P_{0}^m (\varphi)(y)} \lesssim_{s} \frac{\abs{y}^s}{|I|^s}, \qquad y \in I,
$$
and therefore, using the cancellative properties of $f$, we have
\begin{align*}
     \frac{1}{|I|}\absB{\int_\R f(y)\varphi(y)\dd y}  = \frac{1}{|I|}\absB{\int_\R f(y)\brb{\varphi(y)- P_{0}^m (\varphi)(y)}\dd y}\lesssim_s \frac{1}{|I|^{1+s}}.
\end{align*}
Hence, taking the supremum over all $\varphi$ and $I$ and noting that $|I|\gtrsim_m \abs{x}$, we deduce
\begin{align*}
    \MaxAvg^sf(x) \lesssim_s \abs{x}^{-1-s},\qquad x \in \R\setminus(-1,m+3).
\end{align*}
Therefore, since $s>\frac1p-1$, we have
\begin{align*}
    \nrm{\MaxAvg^sf}_{L^p(w_\varepsilon)} \lesssim_s \nrm{\ind_{(-1,m+2)}}_{L^p(w_\varepsilon)}+\nrm{x\mapsto \abs{x}^{-1-s}\ind_{\R\setminus (-1,m+3)}(x)}_{L^p(w_\varepsilon)} \lesssim_{p,s} 1
\end{align*}
proving that $\alpha \geq 1$.
\end{proof}

\begin{remark}\label{rem:nonCZweight}
    Theorem \ref{th:theoremA} and Proposition \ref{prop:weightedsparse} can be combined to give a version of Theorem \ref{th:theoremCZO} for Haar shifts. Indeed, with the same proof (only specializing definitions to the dyadic setting) one can prove that for all $f \in L^\infty_c(\mathbb{R}^d)$ we have
    \begin{align*}
        \|\Sh f\|_{L^p(w)} \lesssim_{s,t} [w]_{A_\infty} \|\sup_Q \per_Q^r(\MaxAvg_Q f)\ind_Q\|_{L^p(w)}
    \end{align*}
    and
    \begin{align} \label{weightedSha}
        \|\Sh f\|_{L^p(w)} \lesssim_{s,t} [w]_{A_q}^{\frac{1}{p}} [w]_{A_\infty}^{(1-\frac{1}{p})_+} \|\MaxAvg f\|_{L^p(w)},
    \end{align}
    where, as in Theorem \ref{th:theoremCZO}, the implicit constant depends on the exponent $q$ such that $w \in A_q$. The right-hand side of \eqref{weightedSha} is one of the equivalent definitions of the dyadic weighted $H^p$-norm from \cite{MR545264} and \cite{ST89}.
\end{remark}

\bibliographystyle{alpha}
\bibliography{bibliography}

\begin{thebibliography}{CCDO17}

\bibitem[BCOR19]{BCOR19}
A.~Barron, J.M. {Conde-Alonso}, Y.~Ou, and G.~Rey.
\newblock Sparse domination and the strong maximal function.
\newblock {\em Adv. Math.}, 345:1--26, 2019.

\bibitem[BFP16]{BFP16}
F.~Bernicot, D.~Frey, and S.~Petermichl.
\newblock Sharp weighted norm estimates beyond {C}alder\'on-{Z}ygmund theory.
\newblock {\em Anal. PDE}, 9(5):1079--1113, 2016.

\bibitem[BG70]{BG1970}
D.~L. Burkholder and R.~F. Gundy.
\newblock Extrapolation and interpolation of quasi-linear operators on
  martingales.
\newblock {\em Acta Math.}, 124:249--304, 1970.

\bibitem[CCDO17]{CCDPO17}
J.M. {Conde-Alonso}, A.~Culiuc, F.~{Di Plinio}, and Y.~Ou.
\newblock A sparse domination principle for rough singular integrals.
\newblock {\em Anal. PDE}, 10(5):1255--1284, 2017.

\bibitem[CF74]{CF74}
R.~R. Coifman and C.~Fefferman.
\newblock Weighted norm inequalities for maximal functions and singular
  integrals.
\newblock {\em Studia Math.}, 51:241--250, 1974.

\bibitem[CF75]{CorFeff75}
A.~C\'ordoba and R.~Fefferman.
\newblock A geometric proof of the strong maximal theorem.
\newblock {\em Annals of Mathematics}, 102(1):95--100, 1975.

\bibitem[CMN19]{CMN19}
D.~{Cruz-Uribe}, K.~Moen, and H.V. Nguyen.
\newblock The boundedness of multilinear {C}alder\'{o}n-{Z}ygmund operators on
  weighted and variable {H}ardy spaces.
\newblock {\em Publ. Mat.}, 63(2):679--713, 2019.

\bibitem[CMP12]{CMP12}
D.V. {Cruz-Uribe}, J.M. Martell, and C.~P{\'e}rez.
\newblock Sharp weighted estimates for classical operators.
\newblock {\em Adv. Math.}, 229(1):408--441, 2012.

\bibitem[CR16]{CR16}
J.M. {Conde-Alonso} and G.~Rey.
\newblock A pointwise estimate for positive dyadic shifts and some
  applications.
\newblock {\em Math. Ann.}, 365(3-4):1111--1135, 2016.

\bibitem[DPS25]{DPS23}
K.~Domelevo, S.~Petermichl, and K.A. Skreb.
\newblock Continuous sparse domination and dimensionless weighted estimates for
  the {B}akry-{R}iesz vector.
\newblock {\em J. Reine Angew. Math.}, 824:137--166, 2025.

\bibitem[DWW23]{DiPlinioWickWilliams2023}
F.~{Di Plinio}, B.D. Wick, and T.~Williams.
\newblock Wavelet representation of singular integral operators.
\newblock {\em Mathematische Annalen}, 386(3):1829--1889, August 2023.

\bibitem[GC79]{MR545264}
Jos\'e Garc\'ia-Cuerva.
\newblock Weighted {H}ardy spaces.
\newblock In {\em Harmonic analysis in {E}uclidean spaces ({P}roc. {S}ympos.
  {P}ure {M}ath., {W}illiams {C}oll., {W}illiamstown, {M}ass., 1978), {P}art
  1}, volume XXXV, Part 1 of {\em Proc. Sympos. Pure Math.}, pages 253--261.
  Amer. Math. Soc., Providence, RI, 1979.

\bibitem[Gra14]{Gra14}
L.~Grafakos.
\newblock {\em Classical {F}ourier analysis}, volume 249 of {\em Graduate Texts
  in Mathematics}.
\newblock Springer, New York, third edition, 2014.

\bibitem[HL22]{HytonenLappas}
T.~Hyt\"onen and S.~Lappas.
\newblock The dyadic representation theorem using smooth wavelets with compact
  support.
\newblock {\em J. Fourier Anal. Appl.}, 28(4), 2022.

\bibitem[HO17]{HO17}
J.~Hart and L.~Oliveira.
\newblock Hardy space estimates for limited ranges of {Muckenhoupt} weights.
\newblock {\em Adv. Math.}, 313:803--838, 2017.

\bibitem[HP13]{HP13}
T.P. Hyt\"onen and C.~P\'erez.
\newblock Sharp weighted bounds involving {$A_\infty$}.
\newblock {\em Anal. PDE}, 6(4):777--818, 2013.

\bibitem[HP16]{HP16}
P.~Hagelstein and I.~Parissis.
\newblock Weighted {Solyanik} estimates for the {Hardy}-{Littlewood} maximal
  operator and embedding of {{\({{\mathcal{A}_\infty}}\)}} into
  {{\({\mathcal{A}_p}\)}}.
\newblock {\em J. Geom. Anal.}, 26(2):924--946, 2016.

\bibitem[Hyt12]{Hy12}
T.P. Hyt\"onen.
\newblock The sharp weighted bound for general {C}alder\'on-{Z}ygmund
  operators.
\newblock {\em Ann. of Math.}, 175(3):1473--1506, 2012.

\bibitem[Lac17]{Lac2017}
Michael~T. Lacey.
\newblock An elementary proof of the {$A_2$} bound.
\newblock {\em Israel J. Math.}, 217(1):181--195, 2017.

\bibitem[Ler13a]{Le13b}
A.K. Lerner.
\newblock On an estimate of {C}alder\'on-{Z}ygmund operators by dyadic positive
  operators.
\newblock {\em J. Anal. Math.}, 121:141--161, 2013.

\bibitem[Ler13b]{Le13a}
A.K. Lerner.
\newblock A simple proof of the {$A_2$} conjecture.
\newblock {\em Int. Math. Res. Not.}, (14):3159--3170, 2013.

\bibitem[LLO22]{LLO21}
A.K. Lerner, E.~Lorist, and S.~Ombrosi.
\newblock Operator-free sparse domination.
\newblock {\em Forum Math. Sigma}, 10:Paper No. e15, 28, 2022.

\bibitem[LN19]{LN15}
A.K. Lerner and F.~Nazarov.
\newblock Intuitive dyadic calculus: {T}he basics.
\newblock {\em Expo. Math.}, 37(3):225--265, 2019.

\bibitem[LOP09]{LOP2009}
A.K. Lerner, S.~Ombrosi, and C.~P\'erez.
\newblock {$A_1$} bounds for {C}alder\'on-{Z}ygmund operators related to a
  problem of {M}uckenhoupt and {W}heeden.
\newblock {\em Math. Res. Lett.}, 16(1):149--156, 2009.

\bibitem[MO11]{MO2011}
Henri Martikainen and Tuomas Orponen.
\newblock A characterization of the boundedness of the median maximal function
  on weighted {$L^p$} spaces, 2011.

\bibitem[NSS26]{NSS24}
Z.~Nieraeth, C.B. Stockdale, and B.~Sweeting.
\newblock Weighted weak-type bounds for multilinear singular integrals.
\newblock {\em The Journal of Geometric Analysis}, 36(178), 2026.

\bibitem[Pet07]{Pe07}
S.~Petermichl.
\newblock The sharp bound for the {Hilbert} transform on weighted {Lebesgue}
  spaces in terms of the classical {{\(A_p\)}} characteristic.
\newblock {\em Am. J. Math.}, 129(5):1355--1375, 2007.

\bibitem[ST89]{ST89}
J.~Str\"{o}mberg and A.~Torchinsky.
\newblock {\em Weighted {H}ardy spaces}, volume 1381 of {\em Lecture Notes in
  Mathematics}.
\newblock Springer-Verlag, Berlin, 1989.

\bibitem[Tom75]{Tomkins75}
R.~J. Tomkins.
\newblock On conditional medians.
\newblock {\em Ann. Probability}, 3:375--379, 1975.

\bibitem[Tom78]{Tom78}
R.J. Tomkins.
\newblock Convergence properties of conditional medians.
\newblock {\em Can. J. Stat.}, 6:169--177, 1978.

\bibitem[Tre13]{Treil2013}
S.~Treil.
\newblock Commutators, paraproducts and {BMO} in non-homogeneous martingale
  settings.
\newblock {\em Rev. Mat. Iberoam.}, 29(4):1325--1372, 2013.

\end{thebibliography}

\end{document}